\documentclass[11pt]{amsart}

\usepackage[margin=1in]{geometry}

\usepackage{amsmath,amssymb,amsthm}

\usepackage{quiver} 
\usepackage{mathrsfs} 
\usepackage{mathtools} 
\usepackage{smileghost} 
\usepackage{halloweenmath}
\usepackage{comment} 
\usepackage{mathtools} 
\usepackage[shortlabels]{enumitem} 

\usepackage[unicode=true, pdfusetitle,
    colorlinks=true,
    linkcolor=orange,
    citecolor=orange,
    urlcolor=orange
]{hyperref} 

\usepackage[nameinlink]{cleveref}

\DeclarePairedDelimiter{\abs}{\lvert}{\rvert}

\DeclareMathOperator{\res}{res}
\DeclareMathOperator{\tr}{tr}
\DeclareMathOperator{\nm}{nm}
\DeclareMathOperator{\conj}{c}
\DeclareMathOperator{\Spec}{Spec}
\newcommand{\uA}{\underline{A}} 
\newcommand{\burn}{\uA} 
\newcommand{\burnghost}[1][G]{\ghost(\burn_{#1})}
\newcommand{\ghostmap}{\chi} 

\newcommand{\mf}[1]{\mathfrak{#1}}
\newcommand{\conjugate}[1]{\sim_{#1}} 
\newcommand{\n}[2]{n_{#1,#2}} 

\newcommand{\NN}{\mathbb{N}}

\newcommand{\ZZ}{\mathbb{Z}}

\newcommand{\I}{\mathcal{I}}

\newcommand{\K}{\mathcal{K}}
\let\P=\relax
\newcommand{\P}{\mathcal{P}}

\newcommand{\sr}[1]{\mathscr{#1}}
\newcommand{\fr}[1]{\mathfrak{#1}}

\newcommand{\Sets}{\mathsf{Set}}
\newcommand{\Fin}{\mathsf{Fin}}

\numberwithin{equation}{section} 
\numberwithin{figure}{section}
\crefname{lemma}{Lemma}{Lemmas}
\crefname{theorem}{Theorem}{Theorems}
\crefname{definition}{Definition}{Definitions}
\crefname{proposition}{Proposition}{Propositions}
\crefname{remark}{Remark}{Remarks}
\crefname{corollary}{Corollary}{Corollaries}
\crefname{equation}{Equation}{Equations}
\crefname{construction}{Construction}{Constructions}
\crefname{ex}{Example}{Examples}
\crefname{appsec}{Appendix}{Appendices}
\crefname{subsection}{Subsection}{Subsections}

\theoremstyle{plain}
\newtheorem{theorem}[equation]{Theorem}
\newtheorem{corollary}[equation]{Corollary}
\newtheorem{proposition}[equation]{Proposition}
\newtheorem{lemma}[equation]{Lemma}
\newtheorem*{theorem*}{Theorem}

\theoremstyle{definition}
\newtheorem{definition}[equation]{Definition}
\newtheorem{example}[equation]{Example}
\newtheorem{remark}[equation]{Remark}

\newtheorem{notation}[equation]{Notation}

\usepackage[
    backend=biber,
    style=alphabetic,
    maxbibnames=99,
    isbn=false,
    url=true,
    doi=false,
    backref=true
]{biblatex}
\DefineBibliographyStrings{english}{%
  backrefpage = {Cited on page},
  backrefpages = {Cited on pages},
}
\DeclareFieldFormat{url}{\url{#1}} 

\usepackage[
    textwidth=3cm,
    colorinlistoftodos]
{todonotes} 

\begin{document}

\title[Spectrum of the Burnside Functor]{The spectrum of the Burnside Tambara functor}

\author[M.E.\ Calle]{Maxine Elena Calle}
\address[Calle]{$\mathghost$ Department of Mathematics,
         University of Pennsylvania,
         Philadelphia, PA, U.S.A.}\email{\href{mailto:callem@sas.upenn.edu}{callem@sas.upenn.edu}}

\author[D.\ Chan]{David Chan}
\address[Chan]{$\bigpumpkin$
Department of Mathematics,
         Michigan State University,
         East Lansing, MI
         U.S.A.}
\email{\href{mailto:chandav2@msu.edu}{chandav2@msu.edu}}

\author[D.\ Mehrle]{David Mehrle}
\address[Mehrle]{$\xleftflutteringbat{}$
        Department of Mathematics,
        University of Kentucky, 
        Lexington, KY, U.S.A.}
\email{\href{mailto:davidm@uky.edu}{davidm@uky.edu}}

\author{J.D. Quigley}
\address[Quigley]{$\bigskull$ Department of Mathematics, University of Virginia, Charlottesville, VA, U.S.A.}
\email{\href{mailto:mbp6pj@virginia.edu}{mbp6pj@virginia.edu}}

\author[B.\ Spitz]{Ben Spitz}
\address[Spitz]{$\xleftswishingghost{}$ Department of Mathematics, Indiana University, Bloomington, IN, U.S.A.}
\email{\href{mailto:bespitz@iu.edu}{bespitz@iu.edu}}

\author[D.\ Van Niel]{Danika Van Niel}
\address[Van Niel]{$\mathwitch$ Department of Mathematics and Statistics, Binghamton University, Binghamton, NY, U.S.A.}
\email{\href{mailto:dvanniel@binghamton.edu}{dvanniel@binghamton.edu}}

\keywords{Burnside Tambara functor, prime ideal, Nakaoka spectrum}

\renewcommand{\subjclassname}{\textup{2020} Mathematics Subject Classification}
\subjclass[2020]{19A22, 
55P91, 
13A15 
}

\begin{abstract}
Tambara functors are an equivariant generalization of commutative rings. In previous work, Nakaoka introduced the spectrum of prime ideals of a Tambara functor and computed the spectrum of the Burnside Tambara functor, the equivariant analogue of the Zariski spectrum of the integers, over cyclic $p$-groups. Subsequently, Calle and Ginnett computed the spectrum of the Burnside Tambara functor over arbitrary finite cyclic groups using a generalization of Dress' ghost coordinates for Burnside rings. 

In this paper, we compute the spectrum of prime ideals in the Burnside Tambara functor over an arbitrary finite group. Our proof uses recent advances in the commutative algebra of Tambara functors, as well as a Tambara functor analogue of ghost coordinates which works over arbitrary finite groups and clarifies some previous computations. As examples, we explicitly compute the spectrum of the Burnside Tambara functor over all dihederal groups,  the quaternion group $Q_8$, the alternating group $A_4$, and the general linear group $GL_3(\mathbb{F}_2)$. 
\end{abstract}

\maketitle

\setcounter{tocdepth}{1}
\tableofcontents

\section{Introduction}

Tambara functors describe multiplicative structures in equivariant stable homotopy theory \cite{BH2015,ElmantoHaugseng} and can be studied in a similar way to ordinary commutative rings --- one can study their ideals \cite{nakaoka:2012}, various kinds of modules \cite{strickland:2012,hill:2017}, and other algebraic structures. The analogue of the ring of integers in this context is called the Burnside Tambara functor, which is the initial object in the category of Tambara functors. Recent progress in equivariant stable homotopy theory, for instance \cite{HHR16, 6AP}, has increased interest in the Burnside Tambara functor since it also arises as the zeroth stable homotopy group of the equivariant sphere spectrum. 

In \cite{nakaoka:2012}, Nakaoka defined prime ideals in Tambara functors. These prime ideals assemble into the Nakaoka spectrum, an equivariant analogue of the Zariski spectrum from classical algebraic geometry. In this paper, we compute the prime ideals in the Burnside Tambara functor over any finite group. This spectrum has previously been computed over cyclic $p$-groups by Nakaoka \cite{nakaoka:2014} and arbitrary cyclic groups in \cite{CalleGinnett:23}.

\begin{theorem}\label{main theorem}
    Let $G$ be a finite group. The Nakaoka spectrum $\Spec(\burn_G)$ of prime ideals in the Burnside $G$-Tambara functor $\burn_G$ is given by
    \[
    	\Spec(\burn_G) = \{\fr{p}_{H,p}\mid H\leq G,\ p \in \mathbb{Z} \text{ prime or } 0\}.
    \]
    The containments between these ideals are listed in \Cref{thm:containment all}. 
\end{theorem}
The ideals $\fr{p}_{H,p}$ are as defined in \cite[Definition 3.1]{calle/ginnett:2019}, recalled in \Cref{def:the_ideals}.
We give examples of the containment structure for nonabelian groups in \Cref{section:examples}.

\begin{remark}
    In \cite[Section 4.2]{nakaoka:2012}, Nakaoka also defined an analogue of the Zariski topology on the spectrum of prime ideals in a Tambara functor, giving it the structure of a topological space. The Burnside Tambara functor $\burn_G$ is levelwise Noetherian and thus Noetherian as a Tambara functor, so by \cite[Thm.\ A]{CMQSV:24}, the topology on $\Spec(\burn_G)$ is determined by the containments listed above.
\end{remark}

Aside from our intrinsic interest in algebraic geometry over Tambara functors, our computations are motivated by a (still conjectural) equivariant analogue of tensor-triangular geometry and its applications in equivariant stable homotopy theory. Classically, the spectrum of thick tensor ideals in the equivariant stable homotopy category is closely related to chromatic phenomena. In \cite{balmer/sanders:2017}, Balmer and Sanders compute the set of thick tensor ideals in the equivariant stable homotopy category for all finite groups by analyzing the map
\[
\operatorname{Spc}(\operatorname{SH}(G)^c) \longrightarrow \operatorname{Spec}(A(G))
\]
from the Balmer spectrum of (the subcategory of compact objects in) the $G$-equivariant stable homotopy category to the Zariski spectrum of the Burnside ring of $G$. The topology on this set is determined for $G=C_p$ in \emph{loc. cit.}, for arbitrary finite abelian groups in \cite{6AP}, and for extra-special $2$-groups in \cite{KL24}. Our computations serve as the target of a conjectured comparison map from an equivariant refinement of the Balmer spectrum of a $G$-tensor-triangulated category to the Nakaoka spectrum of the endomorphisms of the unit; in this case, a conjectured map
\[
\underline{\operatorname{Spc}}(\underline{\operatorname{SH}}(G)^c) \longrightarrow \Spec(\uA_G),
\]
where the underlines indicate the presence of genuine equivariant multiplicative structure.

\subsection{Proof Outline}

Our proof of \Cref{main theorem} employs ideas from \cite{CalleGinnett:23} and \cite{CMQSV:24}. In \cite{CalleGinnett:23}, the first author and Ginnett define prime Tambara ideals $\mathfrak{p}_{H,p}$ of $\uA_G$ for each subgroup $H \leq G$ and $p$ prime or zero (\Cref{def:the_ideals}). Our goal is to show that the ideals $\mathfrak{p}_{H, p}$ are all of the prime ideals in $\uA_G$ over any finite group, and to describe all containments between them.

In \cite{CalleGinnett:23}, the first author and Ginnett use the \textit{ghost} of the Burnside Tambara functor (\Cref{def:ghost}) to understand the poset structure of the Nakaoka spectrum of $\uA_G$. 
The ghost, denoted $\ghost(\uA_G)$, is built out of rings classically known as the tables of marks for subgroups of $G$ and receives a \textit{ghost map} from the Burnside Tambara functor. This map provides ``ghost coordinates'' for working with elements in $\uA_G$.

\begin{remark}
    \label{remark:ghost vs ghost}
    When $G = C_p$, the ghost of the Burnside Tambara functor defined here is isomorphic to the ghost construction of \cite{CMQSV:24} applied to the Burnside Tambara functor $\uA_G$. This motivates our choice of notation. However, the ghost construction of \textit{loc. cit.} is only defined for $G = C_p$; in this paper, we study arbitrary finite $G$, but only consider the ghost of the Burnside Tambara functor rather than a general ghost construction.
\end{remark}

There is an issue with the argument in \cite{CalleGinnett:23} that the ghost of the Burnside Tambara functor is a Tambara functor (see \Cref{remark:cg-issue}), but in \Cref{ghost is TF} and \Cref{prop:ghost is map}, we resolve this issue, confirming that the ghost of the Burnside Tambara functor is indeed a Tambara functor, and moreover, that the ghost map
\[
	\ghostmap: \uA_G \to \ghost(\uA_G)
\]
is a map of Tambara functors. To the best of the authors' knowledge, the table of marks has not appeared as a Tambara functor elsewhere in the literature. Thus, our theorems should be of independent interest.

The ghost map map is a levelwise integral extension, so the induced map
\[
	\ghostmap^*: \Spec(\ghost(\uA_G)) \to \Spec(\uA_G)
\]
is surjective by \cite[Cor. 5.10]{CMQSV:24}. Our problem then reduces to computing $\Spec(\ghost(\uA_G))$ and computing preimages of its contents under $\chi$.

To that end, we produce a family of prime ideals $\mathcal{P}_{H,p} \subset \ghost(\uA_G)$, indexed by pairs $(H,p)$ where $H$ is a subgroup of $G$ and $p$ is a prime number or $0$ (\Cref{thm:prime iff subgp fam}). We then show that these are all of the prime ideals of $\ghost(\uA_G)$ (\Cref{prime over p} and \Cref{prime over zero}). By construction, $\ghostmap^*(\mathcal{P}_{H,p}) = \mathfrak{p}_{H,p}$, where $\mathfrak{p}_{H,p} \subset \uA_G$ is the prime ideal defined in \cite{CalleGinnett:23}. This establishes that $\Spec(\burn_G) = \{\fr{p}_{H,p}\mid H\leq G, p \text{ prime or } 0\}$ as a set, and we subsequently check that the primes $\mathfrak{p}_{H,p}$ satisfy the expected containments (\Cref{thm:containment all}), again using the ghost map.

\subsection{Notation}
Throughout the paper we use the following notation:

\begin{itemize}
	\item $H \leq G$ means $H$ is a subgroup of $G$ and $H < G$ is proper subgroup;
	\item $H\preccurlyeq_G K$ means $H$ is subconjugate to $K$ in $G$ and $H \prec_G K$ means $H$ is properly subconjugate to $K$;
	\item $N \trianglelefteq G$ is a normal subgroup and $N \lhd G$ is a proper normal subgroup;
	\item we write ${}^g\!H := gHg^{-1}$ and $H^g := g^{-1}Hg$;    
	\item for subgroups $H,K \leq G$, we write $H \conjugate{G} K$ if $H$ and $K$ are conjugate subgroups in $G$.
\end{itemize}

\subsection{Acknowledgements}

We would like to thank Tim Dokchitser, who maintains GroupNames (\url{https://people.maths.bris.ac.uk/~matyd/GroupNames/}), a database of finite groups which was incredibly valuable for parts of this work.  We would also like to thank Noah Wisdom for pointing out an omission in the statement of \cref{thm:containment all} in an earlier version of this article. The authors would also like to thank the Isaac Newton Institute for Mathematical Sciences, Cambridge, for support and hospitality during the programme ``Equivariant homotopy theory in context" where work on this paper was undertaken. Finally, the authors would like to thank the anonymous referee for valuable comments and suggestions. 

This work was supported by EPSRC grant no EP/Z000580/1.
MC was partially supported by NSF grant DGE-1845298. DC was partially supported by NSF grant DMS-2135960. DM was partially supported by NSF grant DMS-2135884. JDQ and BS were partially supported by NSF grant DMS-2414922 and JDQ was additionally supported by NSF grant DMS-2441241. DVN was partially supported by NSF grant DMS-2052923.

\section{Background}

\subsection{The Burnside Ring and its Zariski Spectrum}
We recall the work of A.~Dress \cite{Dress1969,dress:1971} on the spectrum of the Burnside ring on a finite group.

\begin{definition} \label{defn:burnside ring}
    The \textit{Burnside ring} of a finite group $G$ is the Grothendieck group completion of the semi-ring of finite $G$-sets, denoted $A(G)$. That is, $A(G)$ is the ring of formal differences of isomorphism classes of finite $G$-sets, with addition given by disjoint union and multiplication given by Cartesian product.
\end{definition}

Let $H$ be a subgroup of $G$, and let $X$ be a finite $G$-set. In \cite{Dress1969}, Dress shows that the
\textit{mark of $H$ on $X$}
\begin{equation}\label{eqn:mark}
    \begin{aligned}
        \varphi_G^H\colon A(G) & \rightarrow \ZZ   \\
        X                      & \mapsto \abs{X^H}
    \end{aligned}
\end{equation}
is a ring homomorphism for all $H\leq G$. Here $X^H$ denotes the $H$-fixed points of $X$.  

Results of Burnside imply that the map $\varphi_G\colon A(G) \rightarrow \prod_{H\leq G} \ZZ$ given by \begin{equation}\label{eqn:phiG}
    \varphi_G = \prod_{H\leq G} \varphi_G^H
\end{equation} 
is an injective ring homomorphism \cite[Lemma 1]{Dress1969}.

\begin{definition}\label{def:phi}
    Let $p \in \ZZ$ be a prime or zero. Let $\varphi_{G, p}^H\colon A(G) \rightarrow \ZZ/p\ZZ$ denote the composite of $\varphi_G^H$ and the usual quotient map $\ZZ \rightarrow \ZZ/p\ZZ$.
\end{definition}

By analyzing the induced map
\[
	\varphi_G^*: \Spec\bigg(\prod_{H \leq G} \mathbb{Z}\bigg) \to \Spec(A(G)),
\]
Dress proved the following.

\begin{theorem}[{\cite[\S 5]{dress:1971}}] \label{thm:spec burn ring}
    The Zariski spectrum of the Burnside ring is $$\Spec(A(G)) = \big\{\!\ker\varphi_{G, p}^H\mid H\leq G, ~p \in \ZZ \text{ a prime or zero}\big\}.$$
\end{theorem}

Dress also proves some containment results which we now recall, following \cite[Theorem 3.6]{balmer/sanders:2017}. We first need a definition. 

\begin{definition}\label{defn:OpH}
    For any subgroup $H \leq G$ and $p$ prime, the \textit{$p$-residual subgroup} of $H$ in $G$ is the unique normal $p$-perfect subgroup $O^p(H)$ such that the quotient $H/O^p(H)$ is a $p$-group. That is,
    \[
        O^p(H) = \bigcap I,
    \]
    where the intersection is over all $I\trianglelefteq H$ such that $|H:I| = p^n$ for some $n\in\NN$.
\end{definition}

\begin{remark}\label{rmk:OpH from sylows}
    One characterization of $O^p(H)$ is as the subgroup generated by all the Sylow $q$-subgroups of $H$ for $q\neq p$. Note that this implies that if $H\preccurlyeq_G K$ then $O^p(H)\preccurlyeq_G O^p(K)$.
\end{remark}
\begin{example}\label{example: Op for p-groups}
    If $G$ is a $p$-group then $O^p(G)=e$ and $O^q(G) =G$ for all $q\neq p$.  
\end{example}
\begin{example}\label{example: Oq for for not q groups}
    If $p$ is any prime which does not divide $|G|$ then $O^p(G)=G$.
\end{example}
\begin{example}
    If $G$ is a non-abelian simple group then $O^p(G)=G$ for all $p$.
\end{example}
\begin{example}\label{example: O of Dpn}
    Let $G=D_{p^n}$ be the dihedral group with $2p^n$ elements for some prime $p>2$ and $n\geq 1$. The only normal subgroups of $D_{p^n}$ are $C_{p^k}$ for $1\leq k\leq n$.  Thus the only non-trivial quotient of $D_{p^n}$ with order a prime power is $C_2 = D_{p^n}/C_{p^n}$.  Thus $O^2(D_{p^n}) = C_{p^n}$ and $O^q(D_{p^n})=D_{p^n}$ for all $q>2$. 
\end{example}
\begin{example}\label{example: O of A4}
    Let $G=A_4$ be the alternating group with $12$ elements.  The only non-trivial normal subgroup of $A_4$ is the Klein $4$-group $K_4$.  Since $A_4/K_4\cong C_3$ we have $O^3(A_4) = K_4$ while $O^q(A_4) = A_4$ for $q\neq 3$. 
\end{example}

Our interest in $O^p(H)$ stems from the following theorem of Dress. Recall that we write $H \sim_G K$ to mean that $H$ and $K$ are conjugate subgroups of $G$.

\begin{theorem}[{\cite{dress:1971}}]\label{thm:dress containment}
    Let $H,K\leq G$ and $p,q$ be primes. Then
    \begin{enumerate}
        \item[(i)] $\ker(\varphi_{G,0}^H)\subseteq \ker(\varphi_{G,0}^{K})$ if and only if $\ker(\varphi_{G,0}^H)= \ker(\varphi_{G,0}^{K})$ if and only if $H\sim_GK$;
        \item[(ii)] $\ker(\varphi_{G,p}^H)\subseteq \ker(\varphi_{G,q}^{K})$ implies $p=q$ and $\ker(\varphi_{G,p}^H)= \ker(\varphi_{G,p}^{K})$;
        \item[(iii)] $\ker(\varphi_{G,p}^H)\subseteq \ker(\varphi_{G,p}^{K})$ if and only if $\ker(\varphi_{G,p}^H)= \ker(\varphi_{G,p}^{K})$ if and only if $O^p(H) \sim_G (O^p(K))$;
        \item[(iv)] $\ker(\varphi_{G,0}^H)\subset \ker(\varphi_{G,p}^{H})$ and $\ker(\varphi_{G,p}^H)\not\subseteq \ker(\varphi_{G,0}^{K})$.
    \end{enumerate}
\end{theorem}

\subsection{Tambara Functors}
\label{SubSec:TambaraFunctors}

In this subsection we give a brief introduction to Tambara functors.  We assume the reader is familiar with the definition of a Mackey functor.  A thorough treatment of both Mackey and Tambara functors can be found in \cite{strickland:2012} or \cite{mazur:2019}. Let $\Sets$ be the category of sets and functions and let $\Fin^{G}$ be the category of finite $G$-sets and $G$-equivariant functions. 

\begin{definition}[{\cite{tambara:1993}}]
\label{def:tf}
A \emph{$G$-Tambara functor} $T$  is a triple $(T^*,T_+,T_\times)$ of functors from $\Fin^G$ to $\Sets$ such that $T^*$ is contravariant and $T_+, T_\times$ are covariant, and all three functors agree on objects: $T^*(U) = T_+(U) = T_\times(U)$ for any finite $G$-set $U$. We write $T(U)$ for the common value of these functors, and given a morphism $f \colon U \to V$ of finite $G$-sets, we write 
\[
	f^* := T^*(f), \quad f_\times := T_\times(f), \quad \text{ and } \quad f_+ := T_+(f),
\]
and call these maps \emph{restriction}, \emph{norm}, and \emph{transfer} along $f$, respectively.
We also require that: 
\begin{enumerate}[(a)]
	\item $(T^*,T_+)$ is a $G$-Mackey functor;
	\item $(T^*,T_\times)$ is a $G$-semi-Mackey functor; 
	\item \label{distributive law tambara functor} given any diagram in $\Fin^G$ isomorphic to one of the form (an \emph{exponential diagram}),
	\[
		\begin{tikzcd}
			Y \ar[d, "h"'] & \ar[l, "g"'] X & Y \times_Z \prod_h X \ar[d, "f"] \ar[l, "e"'] \\
			Z & & \prod_h X, \ar[ll, "k"] 
		\end{tikzcd}
	\]
	we have 
	\[
		k_+ \circ f_\times \circ e^* = h_\times \circ g_+,
	\]
	where $\prod_h$ is the dependent product, right adjoint to the pullback functor along $h$ from the slice category $\Fin^G_{/Z}$ to the slice category $\Fin^G_{/Y}$. See \cite[Section 1]{tambara:1993} for details on the dependent product construction. 
	
\end{enumerate}
\end{definition}

\noindent We expand on this definition below. 

\medskip

The Mackey functor structure on $(T^*, T_+)$ and the semi-Mackey functor structure on $(T^*,T_\times)$ combine to give the structure of a commutative ring on each set $T(U)$, where the distributivity axiom is a consequence of \ref{distributive law tambara functor}. Moreover, the restrictions are ring homomorphisms, norms are multiplicative homomorphisms, and transfers are additive homomorphisms.

Moreover, since $(T^*, T_+)$ is a Mackey functor, the assignment $U \mapsto T(U)$ must send disjoint unions of $G$-sets to products. Since any finite $G$-set $U$ is a disjoint union of its orbits $\big(U \cong \coprod_i G/H_i\big),$ we only need to know the value of $T(G/H)$ for all $H \leq G$ in order to know $T(U)$ for any $U$. 

When $f \colon G/H \to G/K$ is the canonical surjection defined by $f(gH) = gK$, we write 
\[
	\res^K_H := f^*, \quad \nm^K_H := f_\times, \quad \text{ and } \quad \tr^K_H := f_+
\]
for restriction, norm, and transfer along $f$. 
For the conjugation isomorphism $G/H \to G/{}^g\!H$ defined by $aH \mapsto agHg^{-1}$, we write 
\[
	\conj_{g,H} \colon T(G/H) \to T(G/{}^g\!H)
\]
for the restriction along the inverse of this conjugation isomorphism. This is a ring isomorphism for any $g$ and $H$. 

%

It is possible to define a Tambara functor entirely in terms of the structure maps $\res^K_H$, $\tr^K_H$, $\nm^K_H$, and $\conj_{g,H}$. We will need this in \Cref{SS:burnghost}.

\begin{proposition}
\label{Tambara axioms}
Given a collection of commutative rings $\{T(G/H) \mid H \leq G\}$ together with morphisms $\res^K_H$, $\nm^K_H$, $\tr^K_H$, and $\conj_{g,H}$ as above, the following axioms suffice to check that these data define a Tambara functor. 
\begin{enumerate}[(a)]
    \item (Functoriality) 
    	\begin{align*}
			\res^H_L \circ \res^K_H &= \res^K_L	\\
			\nm^K_H\circ \nm^H_L &= \nm^K_L\\
			\tr^K_H\circ \tr^H_L &= \tr^K_L\\
			c_{g,{}^h\!H}\circ c_{g,H} &= c_{gh,H}
		\end{align*} 
    \item\label{conjugacy} (Conjugacy) 
    	\begin{align*}
			c_{g,H}\circ \res^K_H &= \nm^{{}^g\!K}_{{}^g\!H}\circ c_{g,K}\\
			c_{g,K}\circ \nm^K_H &= \nm^{{}^g\!K}_{{}^g\!H}\circ c_{g,H}\\
			c_{g,K}\circ \tr^K_H &= \tr^{{}^g\!K}_{{}^g\!H}\circ c_{g,H}
		\end{align*}
    \item (Additive Double Coset Formula) For any $H,L\leq K$ we have
    \[
        \res^K_L\circ \tr^K_H = \sum\limits_{[\gamma]\in L\backslash K/H} \tr^L_{L\cap {}^\gamma\!H} \circ \res^{{}^\gamma\!H}_{L\cap {}^\gamma\!H}\circ  c_{\gamma}.
    \]
    \item (Multiplicative Double Coset Formula)  For any $H,L\leq K$ we have
    \[
        \res^K_L\circ \nm^K_H({a}) = \prod\limits_{[\gamma]\in L\backslash K/H} \nm^L_{L\cap {}^\gamma\!H} \circ \res^{{}^\gamma\!H}_{L\cap {}^\gamma\!H}\circ  c_{\gamma}({a}).
    \]
    \item (Frobenius Reciprocity) 
    	\[
			\tr_H^K(\res_H^K(x)\cdot y) = x \cdot \tr_H^K(y)
		\]
    \item\label{tambara reciprocity}%
    (Tambara Reciprocity) %
    For any morphism $f : X \to G/H$ of finite $G$-sets, form the exponential diagram
    \[
    	\begin{tikzcd}
			G/H \ar[d, "\phi"'] 
				& 
			X \ar[l, "f"'] 
				& 
			G/H \times_{G/K} \prod_g X 
				\ar[d, "\phi'"] 
				\ar[l, "e"'] 
				\\
			G/K
				& 
				&
			\prod_g X 
				\ar[ll, "p"] 
		\end{tikzcd}
    \]
    in the category $\Fin^G$. Then 
    \[
    	\phi_\times \circ f_+ = p_+ \circ \phi'_\times \circ e^*.
    \]
    To interpret the above equation terms of the data of the maps $\res$, $\nm$, $\tr$, and $\conj$, one chooses for any morphism $h \colon A \to B$ of finite $G$-sets an orbit decomposition of $B$, $B \cong \coprod_i G/L_i$, and pulls this back to an orbit decomposition of $A$ by choosing for each $i$ an orbit decomposition of the preimage of $G/L_i$ under $h$, $h^{-1}(G/L_i) \cong \coprod_j G/J_{ij}$. Finally, 
    \begin{equation}
    \label{generalResNmTr}
    	h^* = \left( \res_{J_{ij}}^{L_i} \right)_{i,j},
		\quad
		h_\times = \bigg(\prod_j \nm_{J_{ij}}^{L_i}\bigg)_i,
		\quad 
		\text{ and } 
		\quad
    	h_+ = \bigg(\sum_j \tr_{J_{ij}}^{L_i} \bigg)_i.
    \end{equation}
    Composing two such morphisms requires one to reconcile two potentially different orbit decompositions using the conjugation isomorphisms. 

\end{enumerate}
\end{proposition}

\subsection{The Burnside Tambara Functor}
\label{SubSec: The Burnside Tambara Functor} 

In this work, we focus on the following Tambara functor:

\begin{definition}\label{defn:burnside tamb functor}
The \emph{Burnside $G$-Tambara functor}, denoted $\burn_G$, is defined by $\burn_G(G/H)=A(H)$ for each $H \leq G$. For subgroups $H\leq K\leq G$, an element $g\in G$, a $K$-set $Y\in \burn_G(G/K)$, and an $H$-set $X\in\burn_G(G/H)$, we define the Tambara structure maps by \begin{align*}
    \res_{H}^{K}\colon \burn_G(G/K)&\longrightarrow \burn_G(G/H)\\
    Y &\longmapsto Y\text{ with restricted $K$-action,}\\
    \tr_{H}^{K}\colon \burn_G(G/H)&\longrightarrow \burn_G(G/K)\\
    X &\longmapsto K\times_H X,\\
    \nm_{H}^{K}\colon \burn_G(G/H)&\longrightarrow \burn_G(G/K)\\
    X &\longmapsto {\rm Map}_H(K,X),\\
    \conj_{g,H}\colon \burn_G(G/H)&\longrightarrow \burn_G(G/\,{}^g\!H)\\
    X &\longmapsto {}^g\!X,
\end{align*}
where ${}^g\!X$ has underlying set $X$ and ${^{g}\!H}$-action $(ghg^{-1})\cdot x = hx$. The restriction, transfers, and conjugation maps may be extended to all finite $G$-sets by linearity. The norm map along any map of finite $G$-sets can be determined using \eqref{generalResNmTr}. 
\end{definition}

\subsection{Tambara ideals, Spectra, and Going Up}
The theory of Tambara ideals mirrors the theory of ideals of commutative rings. An ideal $\I$ of a $G$-Tambara functor $T$ is a levelwise collection of ideals $\I(G/H) \subseteq T(G/H)$ (one for each subgroup of $G$) which is closed under all of the Tambara structure maps; see \cite[Def. 2.1]{nakaoka:2012} for a precise definition. 

The definition below is used to make sense of prime ideals in Tambara functors. 

\begin{definition}
    \label{def:Q}
    Let $T$ be a $G$-Tambara functor, $\I$ an ideal of $T$, and $K_1, K_2\leq G$. Let $a \in T(G/K_1)$ and $b \in T(G/K_2)$. Define the proposition $Q(\I, a, b)$ by
    \begin{enumerate}
        \item[(*)] The relation
              \[
              	\Big(\nm_{{}^{g_{\scalebox{0.6}{\tiny 1}}}\!H_1}^{L}\circ \conj_{g_1, H_1}\circ \res_{H_1}^{K_1}(a)\Big)
                  \cdot\left(\nm_{{}^{g_{\scalebox{0.6}{\tiny 2}}}\!H_2}^{L} \circ \conj_{g_2, H_2} \circ \res_{H_2}^{K_2}(b)\right) \in \I(G/L)\]
              holds for all $L, H_1, H_2 \leq G$ and $g_1, g_2 \in G$ such that $L \geq {}^{g_1}\!H_1, ~L \geq {}^{g_2}\!H_2, ~H_1\leq K_1,$ and $H_2\leq K_2$.
    \end{enumerate}
\end{definition}

Nakaoka proves that the following definition of prime ideals coincides with other possible definitions \cite[Proposition 4.2]{nakaoka:2014}. We use this definition because it is more practical to check.

\begin{definition}
    An ideal $\mathfrak{p}$ of a Tambara functor $T$ is \emph{prime} if
    \begin{enumerate}[(a)]
        \item $\mathfrak{p} \neq T$ (equivalently, for all $H \leq G$, $1 \notin \mathfrak{p}(G/H)$);
        \item Fix any two subgroups $H, H' \leq G$ and elements $a \in T(G/H)$, and $b \in T(G/H')$. If $Q(\mathfrak{p},a,b)$ holds then $a \in \mathfrak{p}(G/H)$ or $b \in \mathfrak{p}(G/H')$.
    \end{enumerate}
\end{definition}

The following proposition will be used to classify prime ideals in the sequel. It implies, in particular, that in a Tambara functor $T$ that has $T(G/e) = \mathbb{Z}$ with trivial $G$-action (the only Tambara functors we will consider below), each prime Tambara ideal $\mf{p}$ has $\mathfrak{p}(G/e) = (p)$ with $p$ a prime integer or zero. 

\begin{proposition}\label{prop:prime at G/e}
    Suppose that $\mathfrak{p}\subset T$ is a prime Tambara ideal such that the Weyl group action of $G$ on $T(G/e)$ is trivial. Then $\mathfrak{p}(G/e)$ is a prime ideal of the commutative ring $T(G/e)$.
\end{proposition}\begin{proof}
The claim follows immediately from \cite[Lemma 3.7]{CMQSV:24}. 
\end{proof}


The \emph{spectrum} of a Tambara functor $T$, denoted $\Spec(T)$, is the collection of all prime ideals, equipped with an analogue of the Zariski topology; see \cite[Section 4.2]{nakaoka:2012}. If $T$ is a Noetherian Tambara functor (e.g., if $T$ is levelwise Noetherian), then the topology on $\Spec(T)$ is completely determined by the poset structure on the set of prime ideals by \cite[Theorem A]{CMQSV:24}. As the Burnside Tambara functor is levelwise finitely generated, and therefore Noetherian, the topology on $\Spec(\burn_G)$ is determined by the poset structure on the collection of prime ideals.

\subsection{The Spectrum over Cyclic Groups}

Nakaoka calculates $\Spec(\burn_G)$ for $G$ a cyclic $p$-group in \cite{nakaoka:2012b, nakaoka:2014}. In \cite{CalleGinnett:23}, the first author and Ginnett extend Nakaoka's computation to all cyclic groups using Dress's computations of the spectrum of the Burnside ring in \cite{dress:1971}.

\begin{definition}\label{def:the_ideals}
    Let $G$ be a finite group and $\burn_G$ the Burnside $G$-Tambara functor. Let $K$ be a subgroup of $G$ and $p \in \ZZ$ a prime or zero. For each $H\leq G$ define
    $$\fr{p}_{K, p}(G/H) = \bigcap_{\substack{I \leq H \\ I \preccurlyeq_G K}} \ker(\varphi_{H, p}^I)$$
    where $\varphi^I_{H,p}$ is as in \Cref{def:phi} and $I \preccurlyeq_G K$ means $I$ is subconjugate to $K$ in $G$.
\end{definition}

For any finite group $G$, the first author and Ginnett show that $\mathfrak{p}_{K,p}$ is a prime ideal of $\burn_G$ for all $K \leq G$ and $p \in \ZZ$ a prime or zero \cite[Theorem 3.8]{CalleGinnett:23}. When $G$ is abelian, they classify pairwise containments between these ideals \cite[Theorem 1.1]{CalleGinnett:23}. Moreover, when $G$ is cyclic, they show in \cite{CalleGinnett:23} that these are \emph{all} of the prime ideals (\emph{op. cit.}).
We show that the primes $\mathfrak{p}_{K,p}$ are \textit{all} of the prime ideals of $\burn_G$ for any finite $G$, and determine all pairwise containments between them in \Cref{thm:containment all}.

\subsection{The Ghost of the Burnside Tambara Functor}

The results of \cite{CalleGinnett:23} are proved using a construction called the \textit{ghost of the Burnside Tambara functor}. This construction is inspired by the \textit{ghost ring} of the Burnside ring
\begin{equation}
\label{eq:table of marks}
	\tilde A(G) := \Big( \prod_{K \leq G} \ZZ \Big)^G,
\end{equation}
the sub-ring of $\prod_{K \leq G} \ZZ$ consisting of the $G$-fixed points, where the $G$-action permutes the coordinates according to the conjugation action of $G$ on its subgroups. This is also known as the \emph{table of marks} of $G$. It is straightforward to verify that the marks homomorphism \eqref{eqn:phiG} factors through $\tilde{A}(G)$.  

\begin{remark}\label{remark: A tilde presentation}
    We could have instead taken the product over \emph{conjugacy classes} of subgroups to obtain the ring $\prod_{[K]} \mathbb{Z}$, which is isomorphic as a ring to $\tilde{A}(G)$. We use the presentation \eqref{eq:table of marks} to eliminate the need to choose representatives of conjugacy classes in many arguments.
\end{remark}

\begin{notation}\label{notation: ideals of A tilde}
    The ideals of $\widetilde{A}(G)$ are precisely the sets $\prod_{K \leq G} n_K \ZZ$ where $(n_K)_{K \leq G}$ is an indexed collection of natural numbers such that $n_{K_1} = n_{K_2}$ whenever $K_1 \conjugate{G} K_2$. We often denote such an ideal simply as the tuple of integers $(n_K)_{K \leq G}$ for brevity.
\end{notation}

Just as the Burnside rings of the subgroups of $G$ assemble into a Tambara functor $\burn_G$, the ghost rings $\widetilde{A}(H)$ also assemble into a Tambara functor. The following definition contains the values of the rings, restriction, transfer, conjugation, and norm maps of this Tambara functor.  The proof that this collection of data satisfies the Tambara axioms, hence actually is a Tambara functor, is deferred to \Cref{SS:burnghost}. 

\begin{definition}\label{cor:structure map fmlas}\label{def:ghost} 
    The \textit{ghost of the Burnside Tambara functor}, denoted $\burnghost$, is defined by $\burnghost(G/H) := \widetilde{A}(H)$. For $H \leq K \leq G$, let $a \in \burnghost(G/H)$, $b \in\burnghost(G/K)$, and $g\in G$. The restriction, transfer, norm, and conjugation maps in $\burnghost$ are given by
    \begin{align*}
        \res_{H}^{K}(b)_L & = b_L                         \\
        \tr_{H}^{K}(a)_I  & = \sum_{\substack{[k] \in K/H       \\ I^k\leq H}}a_{I^k}\\
        \nm_{H}^{K}(a)_I  & = \prod_{[g] \in I\setminus K/H}
        a_{I^g \cap H}                             \\
        \conj_{g, H}(a)_J & = a_{J^{g}}
    \end{align*}
    \noindent for $L \leq H, I \leq K,$ and $J \leq H^g$.

\end{definition}
\begin{remark}
    One can check that $\ghost(\uA_G)$ is isomorphic, as a Green functor, to the \emph{twin construction} considered by Th\'evenaz in \cite[Section 4]{Thev:88}. The Green functor $\ghost(\uA_G)$ also appears in work of Read \cite{Read} who generalizes work of Dress and Siebeneicher \cite{DS88} on the Burnside--Witt construction. While these constructions are much more general than our ghost functor, neither of them are Tambara functors. Indeed, it is unclear how to define norms on either of these constructions that turn them into Tambara functors in general. 
\end{remark}

There is also a natural comparison map from $\uA_G$ to $\burnghost$. The proof that this is a well-defined Tambara functor homomorphism is deferred to \cref{prop:ghost is map}. 

\begin{definition}\label{def:ghostmap}
The \emph{ghost map} 
$$\ghostmap: \uA_G \to \burnghost$$
is defined by setting $\ghostmap(G/H) = \varphi_H$, where $\varphi_H$ is as in \eqref{eqn:phiG}, and extending linearly.
\end{definition}

As mentioned in \Cref{remark:ghost vs ghost}, when $G = C_p$ these constructions coincide (up to isomorphism) with the object $\ghost(\uA_{C_p})$ and the morphism $\chi$ as defined in \cite{CMQSV:24}.

\begin{remark}\label{remark:cg-issue}
    The ghost of the Burnside Tambara functor defined above is the same as in \cite[Definition 3.3]{calle/ginnett:2019}; however, the argument given there that $\burnghost$ is a Tambara functor is incorrect. In particular, the asserted inclusion of $\burnghost$ into the fixed-point Tambara functor on $\prod_{K\leq G} \ZZ$ is \textit{not} a map of Tambara functors as it is not a levelwise ring map (in particular it does not preserve $1$).
    We note that this gap does not affect any of the results of \cite{calle/ginnett:2019}, as the arguments there primarily consider $\burnghost \cap \chi(\burn_G)$, which is indeed a Tambara functor (isomorphic to $\burn_G$) because $\varphi_H$ is injective for all $H\leq G$. Although the proof in \cite{calle/ginnett:2019} contains a gap, $\burnghost$ is nevertheless a Tambara functor (and $\ghostmap$ is a map of Tambara functors), as we show in \Cref{SS:burnghost}.
\end{remark}

\section{Tambara Structure of the Ghost}
\label{SS:burnghost}

In this section, we will prove that the ghost $\ghost(\uA_G)$, defined in \cref{def:ghost}, satisfies the Tambara axioms (\cref{Tambara axioms}) and is indeed a Tambara functor. We begin with four preliminary observations about finite $G$-sets which will be used to check the axioms.

\begin{lemma}[{cf.\ \cite[proof of Lemma 2.18]{CCM}}]\label{lemma: size of fixed points of orbits}
    Let $G$ be a finite group with subgroups $H\leq K\leq G$ and $J\leq G$.  There is a bijection
    \[
        (G/H)^K\cong \coprod_{[x]\in (G/K)^J} (K/H)^{J^x}.
    \]
\end{lemma}

\begin{lemma}\label{lemma: restriction of K/H}
    For any subgroups $H,L\leq K$ there is a bijection
    \[
        \Phi\colon \coprod\limits_{[\gamma]\in L\backslash K/H} L/(L\cap {}^\gamma\!H)\to K/H
    \]
    given explicitly by sending $[\ell]\in L/(L\cap {}^\gamma\!H)$ to $[\ell\gamma]\in K/H$. This bijection is $L$-equivariant with respect to the left action of $L$ on the source and the target.
\end{lemma}
\begin{proof}
	As an $L$-set, $K/H$ is isomorphic to $K \times_H (H/H)$. The result then follows from the Mackey double coset formula. 
\end{proof}

If $J\leq L$ is any subgroup, the map $\Phi$ is a $J$-equivariant bijection.  Taking orbits by the left $J$-action yields the following corollary.

\begin{corollary}\label{corollary: norm double coset formula corollary}
    For any $H,L\leq K$ and $J\leq L$ the map $\Phi$ induces a bijection
    \[
        \coprod\limits_{[\gamma]\in L\backslash K/H} J\backslash L/(L\cap {}^\gamma\!H)\to J\backslash K/H.
    \]
\end{corollary}

We need one final lemma.  Suppose that $H,L\leq K$, and $J\leq L$ is a subgroup such that there is a $k\in K$ with $J^k\subset H$. Note that the set of all $k$ such that $J^k\subseteq H$ is closed under the right $H$-action and let $B\subseteq K/H$ denote the the orbits of this set under this action.  Our goal is to describe the preimage of $B$ under $\Phi$.

For any $\gamma\in L\backslash K/H$, let 
\[
    A_{\gamma} = \{[\ell]\in L/(L \cap {}^\gamma\!H) \mid J^{\ell \gamma}\subset H \}.
\]
By definition, for any $[\ell]\in A_{\gamma}$ we have $\Phi([\ell]) = [\ell\gamma]\in B$.
\begin{lemma}\label{lemma: transfer double coset formula lemma}
    For any $H,L\leq K$ and $J\leq K$ such that there exists $k\in K$ with $J^k\subset H$, if $\Phi$ is the map from \Cref{lemma: restriction of K/H}, then
    \[
        \Phi^{-1}(B) =  \coprod\limits_{[\gamma]\in L\backslash K/H} A_{\gamma},
    \]
    where $B$ and $A_{\gamma}$ are as in the previous two paragraphs.
\end{lemma}
\begin{proof}
    We have already observed that 
    \[
        \coprod\limits_{[\gamma]\in L\backslash K/H} A_{\gamma}\subset \Phi^{-1}(B),
    \]
     so it suffices to show that for any $kH\in B$, there is a $[\gamma]\in L\backslash K/H$ and an $[\ell]\in A_{\gamma}$ so that $\Phi([\ell]) = (\ell\gamma)H = kH$.  By definition of double cosets, there must be a $[\gamma] \in L\backslash K/H$ and an $\ell\in L$ so that $(\ell\gamma) H = kH$, so it remains to check that $[\ell]\in L/(L\cap {}^\gamma\!H)$ is actually in $A_{\gamma}$. That is we need to check that $J^{\ell \gamma}\leq H$. Since $(\ell\gamma) H = kH$ there must be an $h\in H$ such that $\ell\gamma = kh$, hence we have 
     \[
        J^{\ell\gamma} = J^{kh} \leq H^h = H
     \]
     where inclusion $ J^{kh} \leq H^h$ uses the fact that $kH\in B$.
\end{proof}

We now prove a sequence of propositions which will allow us to show that $\burnghost$ is a Tambara functor (\Cref{ghost is TF}). For the rest of the section, the maps $\tr$, $\res$, $\nm$, and $c_\gamma$ always refer to the maps defined in \Cref{def:ghost}. We begin with the functoriality of transfers and norms.

\begin{proposition}[Functoriality of Transfer]\label{proposition:transfer associative}
For any $H \leq L \leq K$ and $a \in \burnghost(G/H) = \big(\prod_{I\leq H}\mathbb{Z}\big)^H$, we have $\tr_L^K \tr_H^L({a}) = \tr_H^K({a}).$
\end{proposition}

\begin{proof}
Recall that the $J$ component of $\tr_L^K( b)$ for some $J \leq K$ is:
\[
	\tr_L^K({b})_J 
		= \sum_{\substack{
					[k] \in K/L \\
					J^k \leq L
				}}
				b_{J^k}.
\]
If $ b = \tr_H^L({a})$, then this becomes:
\[
	\tr_L^K(\tr_H^L({a}))_J 
		= \sum_{\substack{
					[k] \in K/L \\
					J^k \leq L	
			}}
			\tr_H^L( a)_{J^k}
		= \sum_{\substack{
					[k] \in K/L \\
					J^k \leq L	
			}}
			\sum_{\substack{
					[\ell] \in L/H \\
					(J^k)^\ell \leq H
				}} 
				a_{(J^k)^\ell}.
\]
Letting $m = k\ell$, this simplifies to
\[
	\tr_L^K(\tr_H^L({a}))_J 
		= \sum_{\substack{
					[m] \in K/H\\
					J^m \leq H
				}}
				a_{J^m},
\]
which is exactly $\tr_H^L({a})$, as desired. 
\end{proof}

\begin{proposition}[Functoriality of Norm]\label{proposition:norm associative}
For any $H \leq L \leq K$ and ${a} \in \burnghost(G/H) = \big(\prod_{I\leq H}\mathbb{Z}\big)^H$ we have $\nm_L^K \nm_H^L({a}) = \nm_H^K({a}).$
\end{proposition}

\begin{proof}
We have
\begin{align*}
    \nm_H^K({a})_J
     &=  \prod_{[g] \in J \backslash K/H} a_{J^g \cap H} \quad 
    =  \prod_{[g] \in \bigsqcup_{[m] \in J \backslash K / L} (J^m \cap L) \backslash L / H} a_{J^g \cap H} \\
    &=  \prod_{[m] \in J \backslash K/L} \prod_{[g] \in (J^m \cap L) \backslash L / H} a_{J^{mg} \cap H}
\end{align*}
and
\begin{align*}
    \nm_L^K \nm_H^L({a})_J 
    &= \nm_L^K \left( \left( \prod_{[g] \in I \backslash L / H} a_{I^g \cap H} \right)_{I \leq L} \right)_J \\
    &=  \prod_{[m] \in J \backslash K/L} \left( \prod_{[g] \in I \backslash L/H} a_{I^g \cap H} \right)_{J^m \cap L} \\
    &=  \prod_{[m] \in J \backslash K/L} \left( \prod_{[g] \in (J^m \cap L) \backslash L/H} a_{(J^m \cap L)^g \cap H} \right)  \\
    &=  \prod_{[m] \in J \backslash K/ L} \prod_{[g] \in (J^m \cap L) \backslash L/H} a_{J^{mg} \cap L^g \cap H}  \\
    &=  \prod_{[m] \in J \backslash K/ L} \prod_{[g] \in (J^m \cap L) \backslash L/H} a_{J^{mg} \cap H} . \qedhere
\end{align*}
\end{proof}

We now check the double coset formulas.

\begin{proposition}[Additive Double-Coset Formula]\label{proposition: transfer double coset formula}
    For any $H,L\leq K$ and ${a}\in \burnghost(G/H) = \big(\prod_{I\leq H}\mathbb{Z}\big)^H$ we have
    \[
        \res^K_L\tr^K_H({a}) = \sum\limits_{[\gamma]\in L\backslash K/H} \tr^L_{L\cap {}^\gamma\!H} \res^{{}^\gamma\!H}_{L\cap {}^\gamma\!H} c_{\gamma}({a}).
    \]
\end{proposition}
\begin{proof}
    It suffices to check that the claim is true on the component of any subgroup $J\leq L$.  The $J$-th component of the left hand side is given by
    \[
    (\res^K_L\tr^K_H({a}))_J = \tr^K_H({a})_J  = \sum_{\substack{[k] \in K/H       \\ J^k\leq H}} a_{J^k} = \sum_{[k]\in B} a_{J^k},
    \]
    where $B = \{ k \in K: J^k \subseteq H\}/H$ is as above. The $J$-th component of the right hand side is given by
    \begin{align*}
         \sum\limits_{\gamma\in L\backslash K/H} (\tr^L_{L\cap {}^\gamma\!H} \res^{{}^\gamma\!H}_{L\cap {}^\gamma\!H} c_{\gamma}({a}))_J  & =  \sum\limits_{[\gamma]\in L\backslash K/H}\ \sum_{\substack{[\ell] \in L/(L\cap ^{\gamma}H)       \\ J^{\ell}\leq ^{\gamma}H}} a_{J^{\ell\gamma}}\\
         & = \sum\limits_{[\gamma]\in L\backslash K/H}\ \sum\limits_{[\ell]\in A_{\gamma}}\ \, a_{J^{\ell\gamma}}.
    \end{align*}
    As $\bigsqcup_{[\gamma] \in L \backslash K/H} A_\gamma = \Phi^{-1}(B)$ is the inverse image of $B$ under the bijection $\Phi$ by \Cref{lemma: transfer double coset formula lemma}, the two sums are equal. \end{proof}
    
\begin{proposition}[Multiplicative Double-Coset Formula]\label{proposition: norm double coset formula}
    For any $H,L\leq K$ and ${a}\in \burnghost(G/H) = \big(\prod_{I\leq H}\mathbb{Z}\big)^H$ we have
    \[
        \res^K_L\nm^K_H({a}) = \prod\limits_{[\gamma]\in L\backslash K/H} \nm^L_{L\cap {}^\gamma\!H} \res^{{}^\gamma\!H}_{L\cap {}^\gamma\!H} c_{\gamma}({a}).
    \]
\end{proposition}
\begin{proof}
     It suffices to check that the claim is true on the component of any subgroup $J\leq L$.  The $J$-th component of the left hand side is given by
    \[
    (\res^K_L\nm^K_H({a}))_J = \nm^K_H({a})_J  =  \prod_{[k]\in J\backslash K/H} a_{J^k\cap H},
    \]
    while the $J$-th component of the right hand side is given by
    \begin{align*}
         \prod\limits_{[\gamma]\in L\backslash K/H} (\nm^L_{L\cap {}^\gamma\!H} \res^{{}^\gamma\!H}_{L\cap {}^\gamma\!H} c_{\gamma}({a}))_J  & =  \prod\limits_{[\gamma] \in L\backslash K/H} \ \prod\limits_{[\ell]\in J\backslash L/(L\cap ^{\gamma H})} a_{(J^g\cap L\cap ^{\gamma}H)^{\gamma}} \\
         &= \prod\limits_{[\gamma] \in L\backslash K/H}\  \prod\limits_{[\ell]\in J\backslash L/(L\cap ^{\gamma H})}a_{J^{\ell \gamma}\cap H}.
    \end{align*}
    These two products are equal by \Cref{corollary: norm double coset formula corollary}.
\end{proof}

We now check that compatibility of the transfers and multiplicative structure.  First we check Frobenius reciprocity.
\begin{proposition}[Frobenius Reciprocity]\label{proposition Frobenius reciprocity}
    Let $H\leq K\leq G$ be a chain of subgroups.  For any ${a}\in \ghost(\uA_{G})(G/H)$ and ${b}\in \ghost(\uA_G)(G/K)$ we have 
    \[
        \tr^K_H({a})\cdot {b} 
        =
        \tr^K_H({a}\cdot \res^K_H({b})) 
        \quad \mathrm{and} \quad  
         {b} \cdot \tr^K_H({a})
        =
        \tr^K_H( \res^K_H({b})\cdot {a}) 
    \]
\end{proposition}
\begin{proof}
    We check the left relation; the right one follows by commutativity.  It suffices to check that the equation in the statement holds componentwise.  For any subgroup $I\leq K$ we have 
    \[
        \left(\tr^K_H({a})\cdot {b}\right)_I 
        = 
        \tr^K_H({a})_I\cdot {b}_I 
        =
        \sum_{\substack{[k] \in K/H       \\ I^k\leq H}} a_{I^k} \cdot  b_I 
        =
        \sum_{\substack{[k] \in K/H       \\ I^k\leq H}}{a}_{I^k} \cdot {b}_{I^k} 
        =
        \left(\tr^K_H({a}\cdot \res^K_H({b}))\right)_I
    \]
    where the last equality uses the fact that ${b}_{I^k} = {b}_I$ for any $k\in K$ since ${b}\in \ghost(\uA_G)(G/K)$. 
\end{proof}

Finally we check the most difficult relation, Tambara reciprocity, which relates norms and transfers.  
Recall the notation from \cref{def:tf}:  if $f\colon X\to Y$ is an equivariant function between finite $G$-sets, we write $f^*$, $f_{\times}$, and $f_+$ for restriction, norm, and transfer along $f$ in the ghost. In this notation, Tambara reciprocity refers to \cref{Tambara axioms}\ref{tambara reciprocity}.

\begin{proposition}[Tambara Reciprocity]\label{proposition: Tamb Rec}

For any proper subgroup $K < G$ and equivariant functions $t\colon A\to G/K$, and $n\colon G/K\to G/G$ that there is an equality of operations
\begin{equation}\label{equation: exponential equation}
        n_{\times}t_{+} = b_+p_{\times}e^*
    \end{equation}
where $b$, $p$ and $e$ fit into an exponential diagram
\[\begin{tikzcd}
	& E & F \\
	A & {G/K} & {G/G.}
	\arrow["p", from=1-2, to=1-3]
	\arrow["e"', from=1-2, to=2-1]
	\arrow["b", from=1-3, to=2-3]
	\arrow["t"', from=2-1, to=2-2]
	\arrow["n"', from=2-2, to=2-3]
\end{tikzcd}.\]
\end{proposition}

\begin{proof}
Generally, this problem is unwieldy because the exponential diagrams can be difficult to get a handle on.  In this case, the definition of the ghost allows us to check the formula component-wise, i.e., it suffices to check that the formula holds after projecting to the $H$-component of $\burnghost(G/G)$ for every subgroup $H\leq G$. 

First, suppose that $H< G$ is a proper subgroup.  Write $r\colon G/H\to G/G$ for the collapse map.  Since the restrictions in the ghost are given by projections, we see that the formula \eqref{equation: exponential equation} holds on the $H$-component if and only if the equation holds after composition with $r^*$.  That is, it suffices to prove 
\begin{equation}\label{equation: exponential equation 2}
    r^*n_{\times}t_{+} = r^*b_+p_{\times}e^*.
\end{equation}
Consider the diagram
\[
\begin{tikzcd}
	&& C \\
	& E & F & B \\
	A & {G/K} & {G/G} \\
	J & I && {G/H} \\
	& D && L
	\arrow["d"', from=1-3, to=2-2]
	\arrow["{(P)}"{description}, draw=none, from=1-3, to=2-3]
	\arrow["c", from=1-3, to=2-4]
	\arrow["p", from=2-2, to=2-3]
	\arrow["e"', from=2-2, to=3-1]
	\arrow["{(E)}"{description}, draw=none, from=2-2, to=3-3]
	\arrow[""{name=0, anchor=center, inner sep=0}, "b"', from=2-3, to=3-3]
	\arrow["m", from=2-4, to=2-3]
	\arrow[""{name=1, anchor=center, inner sep=0}, "a", from=2-4, to=4-4]
	\arrow["t", from=3-1, to=3-2]
	\arrow[""{name=2, anchor=center, inner sep=0}, "n", from=3-2, to=3-3]
	\arrow["{(P)}"{description}, draw=none, from=3-2, to=4-1]
	\arrow["h", from=4-1, to=3-1]
	\arrow["q"', from=4-1, to=4-2]
	\arrow["g"', from=4-2, to=3-2]
	\arrow[""{name=3, anchor=center, inner sep=0}, "f"', from=4-2, to=4-4]
	\arrow["r"', from=4-4, to=3-3]
	\arrow["j", from=5-2, to=4-1]
	\arrow[""{name=4, anchor=center, inner sep=0}, "k"', from=5-2, to=5-4]
	\arrow["\ell", from=5-4, to=4-4]
	\arrow["{(P)}"{description}, draw=none, from=0, to=1]
	\arrow["{(P)}"{description}, draw=none, from=2, to=3]
	\arrow["{(E)}"{description}, draw=none, from=3, to=4]
\end{tikzcd}
\]
where each polygon is labeled by $(P)$ or $(E)$ corresponding to whether it is a pullback or an exponential diagram in finite $G$-sets.  

The outer maps of the diagram between $A$ and $G/H$ represent two different ways to resolve 
\[
    A \xrightarrow{t} G/K \xrightarrow{n} G/G\xleftarrow{r} G/H
\]
into a TNR morphism in the polynomial category: either $A \xleftarrow{hj} D \xrightarrow{k} L \xrightarrow{\ell} G/H$ or $A \xleftarrow{ed} C \xrightarrow{c} B \xrightarrow{a} G/H$.
Since composition is well-defined in the polynomial category, these two spans are isomorphic hence there are equivariant bijections
\[
    C\leftrightarrow D,\quad B\leftrightarrow L
\]
and through these bijections we can identify 
\[
    hj = ed, \quad k=c, \, \, \, \, \text{and} \quad \ell=a.
\]
Now, since $I$ and $J$ have equivariant maps to $G/K$ and $K$ is proper, neither of these $G$-sets has a fixed point.  In particular, we obtain the formula
\[
    \ell_{+}k_{\times}j^* = f_{\times}q_+
\]
by our inductive hypothesis that the ghost is a Tambara functor for all proper subgroups of $G$.  Using this, and the fact that we have already established the double coset formulas, we have
\begin{align*}
    r^*n_{\times}t_{+}  = f_{\times} g^*t_{+}   = f_{\times}q_+h^* 
     = \ell_+k_{\times}j^*h^*   = a_+c_{\times}d^*e^*
     = a_+m^*p_{\times}e^* = r^*b_+p_{\times}e^*
\end{align*}
establishing the formula \eqref{equation: exponential equation 2} for all proper subgroups $H<G$.

Finally, we need to check that exponential relations hold for the $G$-component.  It suffices, in this case, to establish the result for the two cases $A= G/L$ for $L<K$ a proper subgroup, and for $A= G/K\amalg G/K$ and $t$ the fold map.  In the case $A = G/L$, the left hand side of \eqref{equation: exponential equation} is always zero, since $\nm^G_K({a})_G = a_K$ and $\tr^K_L({b})_K=0$.  By \cite[Theorem 2.8]{mazur:2019}, the right hand side is a sum of proper transfers and hence will also have zero in the $G$-component, thus establishing \eqref{equation: exponential equation} when $A = G/L$.  When $A= G/K\amalg G/K$ we are interested in the formula for $\nm^G_{K}({a}+{b})$.  The $G$-component is given by $a_K+b_K$.  By \cite[Theorem 2.4(2)]{mazur:2019}, the right hand side is given by $\nm^G_{K}({a})+\nm^G_K({b})+T$ where $T$ is a term which is a sum of proper transfers.  In particular, the $G$-component of $T$ is zero, and so the $G$-component of the two sides of \eqref{equation: exponential equation} agree when $A = G/K\amalg G/K$.
\end{proof}

We now have all of the ingredients needed to prove the main theorem of this section. 

\begin{theorem}\label{ghost is TF}
    The ghost of the Burnside Tambara functor $\burnghost$ (\Cref{def:ghost}) is a Tambara functor.  
\end{theorem}
\begin{proof}
    It suffices to check the axioms given in \cref{Tambara axioms}. By induction, we may assume that the claim is true for any proper subgroup $H<G$. The base case, $\tilde{A}(e) \cong \ZZ$, is trivial. Thus we assume that all the necessary formulas which do not involve the group $G$ are satisfied, since these relations are the same as those in the ghost for some maximal proper subgroup of $G$. 

    Functoriality of restrictions and conjugations follows from their definition, while functoriality of transfers and norms follows from \Cref{proposition:transfer associative,proposition:norm associative}, respectively. One can check directly that the conjugacy axiom (\cref{Tambara axioms}\ref{conjugacy}) holds. The double coset formulas for transfer and norm are \Cref{proposition: transfer double coset formula,proposition: norm double coset formula}, respectively. Frobenius reciprocity is checked in \Cref{proposition Frobenius reciprocity} and Tambara reciprocity is shown in \Cref{proposition: Tamb Rec}.
    \end{proof}

With \Cref{ghost is TF} in hand, we can prove the following important fact.

\begin{proposition}\label{prop:ghost is map}
    The ghost map $\ghostmap\colon \burn_G\to \burnghost$ is a morphism of Tambara functors.
\end{proposition}\begin{proof}
    Since the Burnside Tambara functor is initial in the category of Tambara functors, there is a unique map $\Psi_G \colon \burn_G\to \burnghost$ of Tambara functors, which we show coincides with the ghost map. We prove the claim at level $G/G$, but all the other levels follow similarly. For all $H\leq G$, we have
    \[
    \Psi_G(G/H) = \Psi_G(\tr_H^G(H/H)) = \tr_H^G(\Psi_H(H/H)) = \tr_H^G(1),
    \]
    since $\Psi$ commutes with transfers and is a levelwise ring map.  Using the formulas in \Cref{def:ghost}, we see that for any $I\leq G$ 
    \[
    \tr^G_H(1)_I 
    = 
    \sum_{{\gamma\in G/H, ~~I^{\gamma}\leq H}} 1 
    = 
    \abs{(G/H)^I} 
    = 
    \varphi_G^I(G/H),
    \]
    where the second equality uses the bijections
    \[
        (G/H)^I\leftrightarrow \mathrm{Set}^G(G/I,G/H)\leftrightarrow \{\gamma\in G/H\mid I^{\gamma}\leq H\}.
    \]
    Thus $\Psi_G(G/H)  = \ghostmap_G(G/H)$ for every $H\leq G$. Extending linearly, $\Psi_G = \chi_G$ on all of $\burn_G(G/G)$. 
\end{proof}

\begin{remark}
    While we have chosen to present a self-contained proof of \cref{ghost is TF}, the reader familiar with the twin functor construction of \cite{Thev:88} may expedite  the proof by appealing to \cite[Theorem 4.1]{Thev:88} and recognizing that $\burnghost$ agrees with the twin construction of the Burnside functor. This shows that $\burnghost$ is a Green functor; after this shortcut, the rest of the proof of \cref{ghost is TF} only requires checking those Tambara relations involving norms. 
    
    Carefully checking that the twin and the ghost agree, however, amounts to the nearly the same effort as checking directly that $\burnghost$ is a Green functor. We chose to present the latter for ease of exposition.
\end{remark}

\section{Computing the Spectrum of the Burnside Functor}

Inspired by the definition of the ideals $\K_{H,p}$ from \cite[Section 3]{CalleGinnett:23}, we introduce a family of ideals in $\burnghost$. First, recall that a non-empty set of subgroups of $G$ is called a \emph{family of subgroups} if it is closed under conjugation and taking subgroups.

\begin{definition}\label{defn:PFp}
    Let $\sr{F}$ be a family of subgroups of $G$ and $p \in \ZZ$ a prime or $0$. Recalling our \Cref{notation: ideals of A tilde} for ideals of $\tilde{A}$, define $\P_{\sr F, p}\subseteq \burnghost$ by
    \[
    	\P_{\sr F, p}(G/H) = ( \delta_{\sr{F}, p}(I))_{I\leq H}
    \] 
    where
    \[
   		\delta_{\sr{F}, p}(I) = 
		\begin{cases}
        	p & I\in \sr{F}, \\
	        1 & I\not\in \sr{F}.
    	\end{cases}
    \]
\end{definition}

Using the explicit formulas from \Cref{def:ghost}, one can show that $\P_{\sr{F}, p}$ is always an ideal of $\burnghost$, i.e. that it is closed under the structure maps. As a special case, we consider the family $\sr{F}_H = \{I \mid I\preccurlyeq_G H\}$ for some fixed $H\leq G$, in which case we write $\P_{\sr{F}_H, p} = \P_{H,p}$ and similarly $\delta_{\sr{F}_H,p} = \delta_{H,p}$; note that $\P_{H,p}$ only depends on the conjugacy class of $H$ in $G$.

\begin{theorem}\label{thm:prime iff subgp fam}
    The ideal $\P_{\sr{F}, p}$ is prime if and only if $\P_{\sr{F}, p}= \P_{H,p}$ for some $H\leq G$.
\end{theorem}
\begin{proof}
    The ``if'' direction holds by the proof of \cite[Theorem 3.8]{CalleGinnett:23}; although the ideals considered there are $\ghostmap^*\P_{H,p}$, rather than $\P_{H,p}$, the proof in fact establishes that the ideals $\P_{H,p}$ are prime (the intersection with the image of the ghost map only occurs in the final step). 
    
    We show the other implication holds via the contrapositive. Suppose that $\P=\P_{\sr{F},p}$ for a family $\sr{F}$ which is not of the form $\{I \mid I\preccurlyeq_G H\}$ for some $H\leq G$. This means there is some $N,N'\in \sr{F}$ so that for any $L\leq G$ with $N,N'\preccurlyeq_G L$, then $L\not\in \sr{F}$. Consider the elements $a\in \burnghost(G/N)$ and $b\in \burnghost(G/N')$ given by\[
    a_I = \begin{cases}
         0 & I<N, \\
         1 & I=N,
    \end{cases} 
    \quad  \text{ and } \quad 
    b_J = \begin{cases}
         0 & J<N', \\
         1 & J=N',
    \end{cases}
    \] and observe that $a,b \not\in \P$. We show $Q(\P, a,b)$ holds and hence $\P = \P_{\sr{F},p}$ is not prime. Let $K\leq N$, $K' \leq N'$, $g,g'\in G$, and $K^g, K'^{g'}\leq L$ and consider the generalized product \begin{align*}
          &\Big(\nm_{K^g}^{L}\circ \conj_{g, K}\circ \res_{K}^{N}(a)\Big)
                  \cdot\left(\nm_{K'^{g'}}^{L} \circ \conj_{g', K'} \circ \res_{K'}^{N'}(b)\right)\\
                  &= \nm_{K^g}^{L}\Big((a_{I^{g^{-1}}})_{I\leq K^g}\Big)
                  \cdot \nm_{K'^{g'}}^{L}\left((b_{I^{g'^{-1}}})_{I\leq K'^{g'}}\right)\\
                  &= \left(\left(\prod_{\gamma \in J\setminus L/K^g} a_{J^{\gamma g^{-1}}\cap K}\right)\cdot \left(\prod_{\gamma' \in J\setminus L/K'^{g'}} b_{J^{\gamma' g'^{-1}}\cap K'}\right)\right)_{J\leq L}.
    \end{align*} We want to show that this element is in $\P(G/L) = (p_J)_{J\leq L}$. The $J^{th}$ component of this element is 0 unless $J^{\gamma g^{-1}}\cap K =N$ and $J^{\gamma' g'^{-1}}\cap K' = N'$ for all $\gamma, \gamma'$. This means we must have $K=N$ and $K'=N'$, and moreover $N,N'\preccurlyeq_G J$, which then implies that $J\not\in \sr{F}$ and so $1\in (p_J) = \ZZ$, as desired.
\end{proof}

By \cref{prop:prime at G/e}, a prime ideal $\P$ of $\burnghost$ satisfies $\P(G/e)=(p)$ for some $p\in \ZZ$ either prime or zero; we say that $\P$ is a \textit{prime ideal over $p$.}

We prove the main theorem (\Cref{thm:containment all}) by showing that if $\P$ is a prime ideal over $p$, then $\P = \P_{\sr{F}, p}$ for some family $\sr{F}$. \Cref{thm:prime iff subgp fam} then implies that there is some $H\leq G$ so that $\P = \P_{H,p}$. Before doing so, we observe some containment rules.

\begin{proposition}\label{prop:ghost containments}
    For any finite group $G$, subgroups $H,K\leq G$, and $p,q$ prime, there are containments:\begin{enumerate}[(i)]
        \item $\P_{K,0}\subseteq \P_{H,0}$ if and only if $H\preccurlyeq_G K$,
        \item $\P_{H,0}\subseteq \P_{H,p}$ and $\P_{H, p}\not\subseteq \P_{K,0}$, and
        \item $\P_{K,p}\subseteq \P_{H,q}$ if any only if $p=q$ and $H\preccurlyeq_G K$. 
    \end{enumerate}
\end{proposition}\begin{proof}
    The claims are all straightforward to check; we include a proof of the ``only if'' direction of (iii) to indicate the general strategy. Suppose $\P_{K,p}\subseteq \P_{H,p}$. Then at level $G$, we have 
	\[
		\big( \delta_{K,p}(I) \big)_{I \leq G} \subseteq \big( \delta_{H,p}(I) \big)_{I \leq G}
	\]
where $\delta_{K,p}(I)$ is $p$ if $I\preccurlyeq_G K$ and $1$ otherwise. Then $\delta_{H,p}(H) = p$ so $\delta_{K,p}(H)=p$ as well and hence $H\preccurlyeq_G K$.
\end{proof}

We now argue that every prime ideal in $\burnghost$ is of the form $\P_{K,p}$ for some $K\leq G$ and some $p \in \ZZ$ either a prime or $0$. Suppose $\P\in \Spec(\burnghost)$ is a prime ideal over $p$. Since $\P$ is a levelwise ideal, we may write
\[
	\P(G/H) = (\n{I}{H})_{I\leq H}
\] 
for $H\leq G$, $\n{I}{H} \in \ZZ$. In particular, $n_{e,e} = p \in \ZZ$ since $\P$ is a prime over $p$.

\begin{lemma}\label{lem:componentwise prime}
    For all $I\leq H\leq G$, the ideal $(\n{I}{H})\subseteq \ZZ$ is prime or the entire ring.
\end{lemma}
\begin{proof}
	Let $\alpha, \beta\in \ZZ$ so that $\alpha\cdot \beta \in (\n{I_0}{H})$ for some subgroup $I_0 \leq H$ and consider the elements $a, b\in \burnghost(G/H)$ given by 
    \[
    	a_I = 	
		\begin{cases}
        	\alpha & I \conjugate{H} I_0, \\
		    0 & \text{else},
    	\end{cases}
    	\quad \text{ and } \quad 
    	b_I = 
		\begin{cases}
			\beta & I \conjugate{H} I_0, \\
        	0 & \text{else},
		\end{cases}
    \]
    where $\conjugate{H}$ denotes conjugacy in $H$. We want to show that $\alpha$ or $\beta$ is in $(\n{I_0}{H})$, or equivalently, that $a$ or $b$ is in $\P(G/H)$. We do so by showing that $Q(\P, a, b)$ is true, which then implies the desired result as $\P$ is prime.
    
    Let $K_1, K_2\leq H$ and $K_1^{g_1},K_2^{g_2}\leq L$ for $g_1,g_2\in G$. We want to show that the generalized product\[
    \left(\left(\prod_{\gamma \in J\backslash L/K_1^{g_1}} a_{J^{\gamma g_1^{-1}}\cap K_1}\right)\cdot \left(\prod_{\gamma' \in J\setminus L/K_2^{g_2}} b_{J^{\gamma' g_2^{-1}}\cap K_2}\right)\right)_{J\leq L}
    \] is in $\P(G/L)$. The $J^{th}$-component of this product is $0$ (and hence in the ideal) except in the case that $I_0 \preccurlyeq_H K_1,K_2$ and $J^{\gamma g_1^{-1}}\cap K_1 \conjugate{H} J^{\gamma' g_2^{-1}}\cap K_2 \conjugate{H} I_0$ for all $\gamma, \gamma'$. In this case, it suffices to show that\[
    \alpha^{\abs{J\setminus L/K_1^{g_1}}}\cdot \beta^{\abs{J\setminus L/K_2^{g_2}}} \in \P(G/L)_J = (\n{J}{L}).
    \]
    
    Without loss of generality suppose that $\abs{J\backslash L/K_1^{g_1}}\leq \abs{J\backslash L/K_2^{g_2}}$. By assumption, the product $a\cdot b$ is in $\P(G/H)$ and so, since $\P$ is closed under the Tambara structure maps, the element
    \[
        \nm_{K_1^{g_1}}^L  \conj_{g_1, K_1}  \res_{K_1}^H(a\cdot b)
    \] 
    is in $\P(G/L)$. The $J^{th}$-component of this element is \[\prod_{\gamma \in J\setminus L / K_1^{g_1}} (a\cdot b)_{J^{\gamma g_1^{-1}}\cap K_1} = (\alpha\cdot \beta)^{\abs{J\setminus L/K_1^{g_1}}}\] and so $\alpha^{\abs{J\setminus L/K_1^{g_1}}}\cdot \beta^{\abs{J\setminus L/K_1^{g_1}}}\in (\n{J}{L})$ and the claim follows.
\end{proof}

\begin{proposition}\label{prop:value indep of lvl}
    For all $I\leq H\leq G$, the value of $\n{I}{H}$ does not depend on $H$, i.e. $\n{I}{H} = \n{I}{G}$.
 \end{proposition}
 The proof is broken up into three lemmas.

\begin{lemma}
    For any prime ideal $\P$, write $\P(G/K) = (\n{I}{K})_{I\leq K}$ as above. For any subgroup $H\leq G$, the following are equivalent:
    \begin{enumerate}[(a)]
        \item $\n{H}{G}=1$,
        \item $\n{H}{K}=1$ for all $H\leq K\leq G$,
        \item $\n{H}{H}=1$.
    \end{enumerate}
\end{lemma}
\begin{proof}
    The implication $(b)\implies (c)$ is immediate.  The implication $(a)\implies (b)$ follows at once from the fact that the restriction map in the ghost is projection.  Thus it suffices to prove $(c)\implies (a)$.
    
    Let us abbreviate $N = N_GH$. We first show that $\n{H}{N}=1$.  To see this, let $\alpha \in \P(G/H)$ be the element with components
    \[
        \alpha_{I,H} = \begin{cases}
            1 & I=H,\\
            0 & \text{else}.
        \end{cases}
    \]
    Taking the norm we have
    \[
        \nm_H^{N}(\alpha)_{H,N} = \prod\limits_{[g]\in H\backslash N/H} \alpha_{H,H} = 1
    \]
    and thus $1\in (\n{H}{N})$ so $\n{H}{N}=1$.  Let $\alpha'\in \P(G/N)$ be the element with components
    \[ 
        \alpha'_{I,N} = \begin{cases}
            1 & I \conjugate{N} H,\\
            0 & \text{else}.
        \end{cases}
    \]
    which must be in $\P(G/N)$ since $\n{H}{N}=1$.  Taking the transfer we have 
    \[
    \tr^G_N(\alpha')_{H,G} = \sum_{\substack{[\gamma]\in G/N\\H^\gamma \leq N}} \alpha'_{H^{\gamma}, N}. 
    \]
    Let $X = \{[\gamma]\in G/N\mid H^{\gamma} \conjugate{N} H\}$.  Since $\alpha'_{H^{\gamma},N}=0$ unless $\gamma\in X$ we have 
    \[
        \tr^G_N(\alpha')_{H\leq G} = \sum_{\gamma\in X}\alpha_{H^{\gamma},N} = |X|.
    \]
    We claim that $X$ has precisely one element, namely the class of $\gamma = e$, and thus $\tr^G_N(\alpha')_{H, G}=1$ and $\n{H}{G}=1$.  To see this, suppose that $x\in X$ and note that there must be a $y\in N$ such that $H^{xy}=H$.  But then $xy\in N_GH=N$ and hence $x\in N$, so $xN =eN$ as claimed.
\end{proof}
\begin{lemma}
    For any prime ideal $\P$, write $\P(G/K) = (\n{I}{K})_{I\leq K}$ as above. For any subgroup $H\leq G$, the following are equivalent:
    \begin{enumerate}[(a)]
        \item $\n{H}{G}=0$,
        \item $\n{H}{K}=0$ for all $H\leq K\leq G$,
        \item $\n{H}{H}=0$.
    \end{enumerate}
\end{lemma}
\begin{proof}
    The implication $(b)\implies (c)$ is trivial, the implication $(c)\implies (a)$ is immediate from the fact that restriction is given by projection.  Thus it remains to prove $(a)\implies (b)$.  Suppose there is a $K\leq G$ with $\n{H}{K} \neq 0$.  Since restriction is projection it follows that $n = \n{H}{H} \neq 0$.  Consider the element $\alpha\in \P(G/H)$ with components
    \[
        \alpha_{I, H} = 
        \begin{cases}
            n & I=H,\\
            0 & \text{else}.
        \end{cases}
    \]
    It is immediate from the definitions that $\tr^G_H(\alpha)_{H\leq G}\neq0$.  Since $\tr^G_H(\alpha)_{H\leq G}\in (\n{H}{G})$ we have $\n{H}{G} \neq 0$.  The implication $(a)\implies (b)$ follows by contrapositive.\end{proof}

\begin{lemma}
    Let $p$ be a prime number. For any prime ideal $\P(G/K) = (\n{I}{K})_{I\leq K}$ and any subgroup $H\leq G$ the following are equivalent:
    \begin{enumerate}[(a)]
        \item $\n{H}{G}=p$,
        \item $\n{H}{K}=p$ for all $H\leq K\leq G$,
        \item $\n{H}{H}=p$.
    \end{enumerate}
\end{lemma}
\begin{proof}
    Since $\P$ is prime we know $\n{H}{G}$ and $\n{H}{H}$ are prime numbers, $0$, or $1$.  By the previous two lemmas, we know that if either or these numbers is $0$ or $1$ then all three of the possibilites in the statement of the lemma are impossible. So without loss of generality we assume that $\n{H}{G}$ and $\n{H}{H}$ are both prime numbers, though not necessarily the same ones.
    
    The implication $(b)\implies (c)$ is trivial, the implication $(c)\implies (a)$ follows from the fact that restriction is a projection which implies that $\n{H}{G}\in (\n{H}{H}) = (p)$. Since $\n{H}{G}$ is prime we see that $\n{H}{G}=p$.

    To prove that $(a)\implies (b)$, suppose that $\n{H}{G}=p$ and there is some $K\leq G$ with $\n{H}{K}\neq p$.  Since restriction is projection we have $\n{H}{G}=p\in(\n{H}{K})$ which implies that $(\n{H}{K})$ is either $1$ or $p$.  If $\n{H}{K}=1$ then restriction implies $\n{H}{H}=1$, but we have already seen that this would imply $\n{H}{G}=1$ which is a contradiction.  Thus $(\n{H}{K})=p$ and $(a)\implies (b)$.
\end{proof}

We may therefore write $\P(G/H) = (p_I)_{I\leq H}$ where each $p_I$ is $0$, $1$, or prime. Recall that if $\P(G/e) = (p) \subseteq \ZZ$, then we say that $\P$ is a \textit{prime ideal over $p$}.

\begin{lemma}
    \label{lemma: families from primes}
    If $\P$ is a prime ideal over $p$, then $\sr{F}(\P) := \{I\leq G\mid p_I\neq 1\}$ is a nonempty family of subgroups.
\end{lemma}
\begin{proof}
    For $H\leq K$, the $K$-component of $\nm_H^K((n_I)_{I\leq H})$ is $\prod_{K\setminus K/H} n_{K\cap H} = n_H$. This implies $n_H\in (n_K)$, and hence if $n_K\neq 1$, we must have $n_H\neq 1$.
\end{proof}

\begin{remark}
    The exact same argument shows that $\sr{F}_0 := \{I\leq H\mid p_I=0\}$ is a family of subgroups. This will only be relevant for $\P$ a prime ideal over $0$.
\end{remark}

Recall from \Cref{defn:PFp} that for a family of subgroups $\sr{F}$, we consider the ideal $\P_{\sr{F},p}(G/H) = (p_I)_{I \leq H}$ where $p_I$ is $p$ if $I \in \sr{F}$ and $1$ else.

\begin{corollary}\label{prime over p}
    If $\P$ is a prime ideal over $p\neq 0$, and $\sr{F}(\P)$ is as in \Cref{lemma: families from primes},  then $\P=\P_{\sr{F}(\P),p}$.
\end{corollary}
\begin{proof}
    By assumption, $p_e=p\neq 0$. Taking the norm 
    \[
    	\nm_e^G(p_e) = \bigg(\prod_{\gamma\in I\setminus G/e} p_{I^{\gamma}\cap e}\bigg)_{I\leq G} = \big(p^{\abs{G:I}}\big)_{I\leq G}
    \] shows $p^{\abs{G:I}}\in (p_I)$ for all $I\leq G$. Thus $p_I$ is either $1$ or $p$, so $\P$ is as claimed.
\end{proof}

\begin{proposition}\label{prime over zero}
    If $\P$ is a prime ideal over $0$, then $\sr{F}(\P) = \sr{F}_0$ and $\P = \P_{\sr{F}(\P), 0}$.
\end{proposition}
\begin{proof}
    Suppose for contradiction that the containment $\sr{F}_0\subseteq \sr{F}(\P)$ is proper. Define \[
    n = \prod_{H\in \sr{F}\setminus \sr{F}_0} p_H
    \] and choose any $\hat I\in \sr{F}_0$ maximal, i.e. with the property that if $\hat I\preccurlyeq_G H$ then $H\not\in \sr{F}_0$. Consider the element $a\in \burnghost(G/\hat I)$ given by
    \[
    	a_I = 
		\begin{cases}
        	n &  I = \hat I,\\
	        0 & \text{else},
		\end{cases}
    \] 
    and observe $a\not\in \P$. We show that $Q(\P, a,a)$ holds, contradicting the fact that $\P$ is prime, and consequently we may conclude that $\sr{F}_0 = \sr{F}(\P)$. It follows immediately that $\P = \P_{\sr{F}(\P), 0}$.

    Let $K_1, K_2\leq H$, $g_1,g_2\in G$, and $K_1^{g_1}, K_2^{g_2}\leq L$. We want to show that the generalized product\[
    \left(\left(\prod_{\gamma \in J\setminus L/K_1^{g_1}} a_{J^{\gamma g_1^{-1}}\cap K_1}\right)\cdot \left(\prod_{\gamma' \in J\setminus L/K_2^{g_2}} a_{J^{\gamma' g_2^{-1}}\cap K_2}\right)\right)_{J\leq L}
    \] is in $\P(G/L)$. The $J^{th}$-component of this product is $0$ (and hence in the ideal) except in the case that $\hat I \preccurlyeq_H K_1,K_2$ and $J^{\gamma g_1^{-1}}\cap K_1 \conjugate{H} J^{\gamma' g_2^{-1}}\cap K_2 \conjugate{H} \hat I$ for all $\gamma, \gamma'$. In this case, $\hat I\preccurlyeq_G J$ and the generalized product is $n^{\abs{J\setminus L/K_1^{g_1}}+\abs{J\setminus L/ K_2^{g_2}}}$. But then $J\not\in \sr{F}_0$ by our choice of $\hat I$, and so $n\in (p_J)$ by the definition of $n$.
\end{proof}

Recall the prime ideals $\fr{p}_{H,p}$ of $\burn_G$ from \cref{def:the_ideals}. 

\begin{corollary}
    For any finite group $G$,
    \[
	    \Spec(\burn_G) = 
    	\{\fr{p}_{H,p}\mid H\leq G,\ p \in \ZZ \text{ prime or } 0\}.    
    \]
\end{corollary}

\begin{proof}
    We have shown that the prime ideals of $\burnghost$ are precisely the $\P_{H,p}$. Since the ghost map is a levelwise integral extension, by \cite[Corollary 5.7]{CMQSV:24}, the ghost map induces a surjection on $\Spec$, which is precisely the assignment $\P_{H,p}\mapsto \fr{p}_{H,p}$.
\end{proof}

We now turn to the question of the poset structure on $\Spec(\burn_G)$.  That is, we describe when there are containments $\mf{p}_{H,p}\subseteq \mf{p}_{K,q}$ for $H,K\leq G$ and $p$ and $q$ prime or $0$. In \cite[Theorem 3.9]{CalleGinnett:23}, the authors establish when some containments occur for abelian $G$. We extend this result to apply to any finite group in \cref{thm:containment all} below.  Before doing so, we need a few preliminary lemmas.

Let $H$ be a $p$-perfect subgroup of $G$, i.e.\ $H = O^p(H)$, where $O^p(H)$ is the $p$-residual subgroup as in \cref{defn:OpH}. Let $a\in \widetilde{A}(G)$ be the element of the ghost of $G$ defined by
\[
    a_I = \begin{cases}
        1 & O^p(I) \sim_G H,\\
        0 & \text{else.}
    \end{cases}
\]
Note that by assumption $O^p(H) = H$ and so $a_H=1$. This is indeed an element in the ghost, since the value of $a_I$ depends only on $I$ up to conjugacy. 

\begin{lemma}\label{lemma: p local Burnside element}
    There exists an integer $n>0$, coprime to $p$, such that $n\cdot a$ is in the image of the ghost map.
\end{lemma}
\begin{proof}
    Consider the commutative diagram
    \[
    \begin{tikzcd}
        A(G) \ar[d,swap,"\psi",hook] \ar[r,"\chi",hook] & \widetilde{A}(G) \ar[d,"\widetilde\psi",hook]\\
        A_{(p)}(G) \ar[r,swap,"\chi_{(p)}",hook] & \widetilde{A}_{(p)}(G)
        \end{tikzcd}
    \]
    where the subscript denotes localization at $(p)$, i.e.\ all primes except for $p$ are made invertible.  The vertical maps are inclusions, and the horizontal maps are both injective.
    By \cite[Remark 1.4.3]{tomDieck}, there is some $x\in A_{(p)}(G)$ such that $\chi_{(p)}(x) = \widetilde{\psi}(a)$.  On the other hand, there must be some $n$ coprime to $p$ so that $n\cdot x = \psi(y)$ for some $y\in A(G)$, hence
    \[
        \widetilde{\psi}(n\cdot a) = n\chi_{(p)}(x) = \chi_{(p)}(nx) = \chi_{(p)}(\psi(y)) = \widetilde{\psi}(\chi(y))
    \]
    and since $\widetilde{\psi}$ is injective we must have $\chi(y) = n \cdot a$.
    \end{proof}

\begin{lemma}\label{lem:pHp=pOpHp}
    For $H\leq G$ and $p\in \ZZ$ prime, there is an equality $\fr{p}_{H,p} = \fr{p}_{O^p(H), p}$.
\end{lemma}\begin{proof}
     Since the ghost map $\chi \colon \burn_G \hookrightarrow \burnghost$ is injective, it suffices to show that $\K_{H,p} = \K_{O^p(H), p}$, where $\K_{H,p}$ is defined to be the intersection of $\P_{H,p}$ with the image of the ghost map, 
     \[
     	\K_{H,p}(G/L) := \P_{H,p}(G/L) \cap \chi(\burn_G(G/L)). 
     \]  
     Recall also that $\P_{H,p}(G/L) = \big(\delta_{H,p}(I)\big)_{I \leq L}$, where $\delta_{H,p}(I)$ is $p$ if $I\preccurlyeq_G H$ and $1$ otherwise.
     
     First observe that the $\subseteq$ containment is immediate. For $\supseteq$, suppose we have $X\in A(G)$ so that $\varphi^I(X) = \abs{X^I}\in (p)$ for all $I\preccurlyeq_G O^p(H)$. We want to show that $\varphi^J(X) = \abs{X^J}\in (p)$ for all $J\preccurlyeq_G H$, but $\abs{X^J} \equiv_p \abs{X^{O^p(J)}}$ (c.f. \cite[Remark 2.18] {CalleGinnett:23}). Hence it suffices observe that if $J\preccurlyeq_G H$, then $O^p(J)\preccurlyeq_G O^p(H)$, which is \Cref{rmk:OpH from sylows}.
\end{proof}

\begin{theorem}\label{thm:containment all}
    Let $H,K\leq G$ and $p,q$ be prime. Then
    \begin{enumerate}
    \item[(i)] $\fr{p}_{K, 0}\subseteq \fr{p}_{H, 0}$ if and only if $H \preccurlyeq_G K$,
    \item[(ii)] $\fr{p}_{H, 0}\subset \fr{p}_{H, p}$ and $\fr{p}_{H, p}\not\subseteq \fr{p}_{K, 0}$,
    \item[(iii)] $\fr{p}_{K, p}\subseteq \fr{p}_{H, q}$ if and only if $p = q$ and $O^p(H)\preccurlyeq_G O^p(K)$,
    \item[(iv)] $\mf{p}_{K,0}\subseteq \mf{p}_{H,p}$ if and only if $O^p(H)\preccurlyeq_G K$.
\end{enumerate} where $O^p(H)$ is the $p$-residual subgroup as in \Cref{defn:OpH}.
\end{theorem}
\begin{proof}
As before, it suffices to prove the analogous claims about the $\K_{L, p}$.

The ``if'' direction of (i) is immediate from \Cref{prop:ghost containments}. For the ``only if,'' we show the contrapositive. Suppose $H\not\preccurlyeq_G K$, so the element $b\in \burnghost(G/H)$ given by
\begin{equation}\label{equation: definition of element a}
	b_I = 
	\begin{cases}
    	1 & I = H, \\
	    0 & \text{else},
	\end{cases}
\end{equation}
is in $\P_{K,0}(G/H)$ but not in $\P_{H,0}(G/H)$. Since the ghost map has finite cokernel \cite[Proposition 1.2.3]{tomDieck}, there is some $n \geq 0$ so that $n\cdot b$ is in the image of $\chi_{H}\colon A(H)\to \burnghost(G/H)$. This means that $n\cdot b\in \K_{K,0}(G/H)$ but not in $\K_{H,0}(G/H)$, so then $\K_{K,0}\not\subseteq \K_{H,0}$ as desired.

The first part of (ii) is also immediate from \Cref{prop:ghost containments}. For the second part, we simply observe that $\K_{H,p}(G/e) = (p)\subset \mathbb{Z}$ while $\K_{K,0}(G/e) = (0)\subset \mathbb{Z}$. We have used here that the ghost map is an isomorphism at level $G/e$.

For (iii), we observe that the ``if'' direction is immediate from (i) and \cref{lem:pHp=pOpHp}. To prove the ``only if'' direction, suppose that $\mf{p}_{K,p}\subseteq \mf{p}_{H,q}$. At the $G/e$ level, we obtain $(p)\subseteq (q)$ and hence $p=q$.  Next, observe that since  $\K_{H,p} = \K_{O^p(H),p}$, by \cref{lem:pHp=pOpHp}, we may reduce to case that $K=O^p(K)$ and $H = O^p(H)$ are both $p$-perfect. 

Consider the element $a\in \burnghost(G/G)$ given by\[
    a_I = \begin{cases}
        1 & O^p(I) \sim_G H,\\
        0 & \mathrm{else}.
    \end{cases}
\] Note that by assumption $O^p(H) = H$ and so $a_H=1$. We claim that there is an integer $n>0$ which is coprime to $p$ such that $n\cdot a$ is in the image of the ghost map; this is shown in \cref{lemma: p local Burnside element}. Then the element $n\cdot a$ is not an element in $\mathcal{K}_{H,p}(G/G)$, since $(n\cdot a)_H = n\cdot a_H = n$ is coprime to $p$.  Suppose that $H$ is not subconjugate to $K$.  If $I$ is any subgroup of $K$ then $O^p(I)$ is not conjugate to $H$, as this would imply that $H\sim_G O^p(I)\leq I\leq K$.  Thus we see that $a_I=0$ for all $I\leq K$, and hence $a_I\in \mathcal{K}_{K,p}$.  Since we assumed $\mathcal{K}_{K,p}\subseteq \mathcal{K}_{H,p}$, this is a contradiction, hence we must have $H\preccurlyeq_G K$. 

For (iv), the ``if'' direction follows immediately from (i), (ii), and \cref{lem:pHp=pOpHp}.  The ``only if'' direction is essentially the same proof as (iii), except that we may not assume $K = O^p(K)$.
\end{proof}

\begin{corollary}
    If $G$ is a $p$-group then $\mf{p}_{H,p} = \mf{p}_{K,p}$ for all $H,K\leq G$.  For $q\neq p$ we have $\mf{p}_{H,q} = \mf{p}_{K,q}$ if and only if $H$ and $K$ are conjugate.
\end{corollary}
\begin{proof}
    It is immediate from (iii) above that $\mf{p}_{H,p} = \mf{p}_{K,p}$ if and only if $O^p(H)$ and  $O^p(K)$ are conjugate. For the first claim, since $H$ and $K$ are both $p$-groups both of the $p$-residual subgroups are the trivial group (\Cref{example: Op for p-groups}), hence are equal.  For the second claim, note that we have $O^q(A)=A$ for all subgroups $A\leq G$ (\Cref{example: Oq for for not q groups}).
\end{proof}

The Krull dimension of a Tambara functor is defined similarly to the Krull dimension of a commutative ring \cite[Def. 8.1]{CMQSV:24}. The results above imply that the Krull dimension of $\burn_G$ is always one more than the length of the longest chain in the poset of conjugacy classes of subgroups of $G$; the zero ideal $\mf{p}_{e,0}$ is contained in every other prime $\mf{p}_{H,p}$ and accounts for the plus one.

\begin{proposition}
    For any finite group $G$, there is a continuous bijection
    \[\Spec(A(G)) \to \Spec(\uA_G),\]
   but $\Spec(A(G))$ and $\Spec(\uA_G)$ are not homeomorphic unless $G$ is the trivial group.
\end{proposition}
\begin{proof}
    We have established that
    \[\ker \varphi_{G,p}^H \mapsto \mf{p}_{H,p} : \Spec(A(G)) \to \Spec(\uA_G)\]
    is a bijection. By \Cref{thm:containment all} and \Cref{thm:dress containment}, this is a morphism of posets. Since $A(G)$ is a Noetherian ring and $\uA_G$ is a Noetherian Tambara functor, this function is continuous by \cite[Theorem 2.24]{CMQSV:24}.

    To see that $\Spec(A(G))$ and $\Spec(\uA_G)$ are not homeomorphic unless $G$ is trivial, let $q$ be any prime number not dividing the order of $G$ and suppose $G \neq e$. Then we have proper containments
    \[\mf{p}_{e,0} \subset \mf{p}_{G,0} \subset \mf{p}_{G,q},\]
    demonstrating that the Krull dimension of $\uA_G$ is at least $2$. On the other hand, the Krull dimension of $A(G)$ is $1$.
\end{proof}

\section{Examples}\label{section:examples}

In this section we give explicit descriptions of $\Spec(\uA_G)$ for various groups $G$.  We focus on non-abelian groups since these are examples which have not previously appeared in the literature.  To describe these spaces, we note that there is a poset map $\Spec(\uA_G)\to \Spec(\mathbb{Z})$ given by $\mf{p}_{H,p}\mapsto (p)$ for $p \in \ZZ$ a prime or $0$.  We describe the poset $\Spec(\uA_G)$ by describing the fiber of this map over every prime $(p)$.  Property (iii) from \Cref{thm:containment all} implies there are no containments between prime ideals which live in fibers over distinct primes, and properties (ii) and (iv) describe possible containments between the fiber over $(0)$ and the fibers over primes.

We will frequently make use of \Cref{example: Op for p-groups} and \Cref{example: Oq for for not q groups}, which say that if $G$ is a $p$-group, then $O^p(G) = e$ and $O^q(G) = G$ for $q \neq p$, and if $p$ is any prime which does not divide $|G|$, then $O^p(G)=G$. In particular, \Cref{example: Oq for for not q groups} implies that condition (iii) in \Cref{thm:containment all} becomes $\mf{p}_{H,p}\subset \mf{p}_{K,p}$ if and only if $K\preccurlyeq_G H$.  Thus, at any generic prime $p$ the fiber is simply the opposite of the subconjugacy poset, i.e.\ the poset obtained by taking orbits of the subgroup poset by the conjugation action.  

In all our examples we begin by displaying the subconjugacy poset, and then  display the poset corresponding to primes which do divide the group order. If $x,y$ are elements of a poset $P$ we write $x\to y$ to indicate that $x\leq y$.

\begin{example}
    Let $G = D_{p^n}$ be the dihedral group of order $2p^n$ where $p\neq 2$ is a prime.
    The fiber over a generic prime $q\neq 2,p$ is the opposite of the subconjugacy poset:
\[\begin{tikzcd}
	{\mf{p}_{D_1,q}} & {\mf{p}_{D_p,q}} & \dots & {\mf{p}_{D_{p^{n-1}},q}} & {\mf{p}_{D_{p^n},q}} \\
	{\mf{p}_{e,q}} & {\mf{p}_{C_{p},q}} & \dots & {\mf{p}_{C_{p^{n-1}},q}} & {\mf{p}_{C_{p^n},q}} .
	\arrow[to=1-1, from=1-2]
	\arrow[to=1-2, from=1-3]
	\arrow[to=1-3, from=1-4]
	\arrow[to=1-4, from=1-5]
	\arrow[to=2-1, from=1-1]
	\arrow[to=2-1, from=2-2]
	\arrow[to=2-2, from=1-2]
	\arrow[to=2-2, from=2-3]
	\arrow[to=2-3, from=1-3]
	\arrow[to=2-3, from=2-4]
	\arrow[to=2-4, from=1-4]
	\arrow[to=2-4, from=2-5]
	\arrow[to=2-5, from=1-5]
\end{tikzcd}\]
The fiber over $q=2$ is controlled by the computations of $O^2$ of subgroups of $D_{p^n}$.  Since $C_{p^k}$ is a $p$-group, and $p\neq 2$, we have $O^2(C_{p^k}) = C_{p^k}$ for all $k$.  On the other hand we have, by \Cref{example: O of Dpn}, that $O^2(D_{p^k}) = C_{p^k}$ as well so the fiber over $q=2$ is a line
\[\begin{tikzcd}
	{\mf{p}_{e,2}} & {\mf{p}_{C_{p},2}} & \dots & {\mf{p}_{C_{p^{n-1}},2}} & {\mf{p}_{C_{p^n},2}} 
	\arrow[to=1-1, from=1-2]
	\arrow[to=1-2, from=1-3]
	\arrow[to=1-3, from=1-4]
	\arrow[to=1-4, from=1-5]
\end{tikzcd}.\]
For the fiber over $q=p$ we have that $O^p(C_{p^k})=e$ and $O^p(D_{p^k}) = D_{p^k}$ for all $k$.  Thus the fiber over $q=p$ is 
\[\begin{tikzcd}
	{\mf{p}_{e,p}} & {\mf{p}_{D_1,p}} & {\mf{p}_{D_{p},p}} & \dots & {\mf{p}_{D_{p^{n-1}},p}} & {\mf{p}_{D_{p^n},p}}
	\arrow[to=1-1, from=1-2]
	\arrow[to=1-2, from=1-3]
	\arrow[to=1-3, from=1-4]
	\arrow[to=1-4, from=1-5]
	\arrow[to=1-5, from=1-6]
\end{tikzcd}.\]
Note that while the fibers over $q=2$ and $q=p$ are both lines, the lengths of these lines differ by one.
\end{example}

\begin{example}
    Let $G = Q_8$.  The fiber over a generic prime $p\neq 2$ is the opposite of the subgroup lattice:
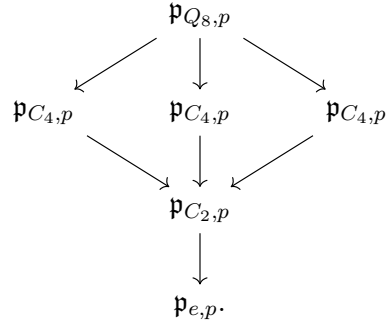
\[\begin{tikzcd}
	& {\mf{p}_{Q_8,p}} \\
	{\mf{p}_{C_4,p}} & {\mf{p}_{C_4,p}} & {\mf{p}_{C_4,p}} \\
	& {\mf{p}_{C_2,p}} \\
	& {\mf{p}_{e,p}}.
	\arrow[to=2-1, from=1-2]
	\arrow[to=2-2, from=1-2]
	\arrow[to=2-3, from=1-2]
	\arrow[to=3-2, from=2-1]
	\arrow[to=3-2, from=2-2]
	\arrow[to=3-2, from=2-3]
	\arrow[to=4-2, from=3-2]
\end{tikzcd}\]
Since $G$ is a $2$-group the fiber over $p=2$ is a single point (\Cref{example: Op for p-groups}). Completely analogous arguments handle the generalized quaternion case $G = Q_{4n}$ for all $n \in \mathbb{N}$. 
\end{example}

\begin{example}
    Let $G = A_4$ be the alternating group with $12$ elements. The fiber over a generic prime $q\neq 2,3$ is the opposite of the subconjugacy poset
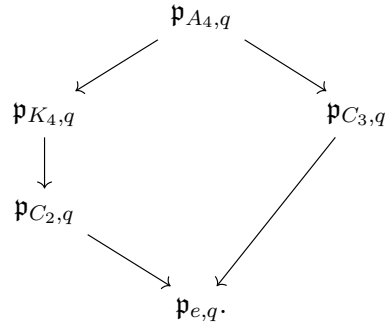
\[\begin{tikzcd}
	& {\mf{p}_{A_4,q}} \\
	{\mf{p}_{K_4,q}} && {\mf{p}_{C_3,q}} \\
	{\mf{p}_{C_2,q}} \\
	& {\mf{p}_{e,q}}.
	\arrow[to=2-1, from=1-2]
	\arrow[to=2-3, from=1-2]
	\arrow[to=3-1, from=2-1]
	\arrow[to=4-2, from=2-3]
	\arrow[to=4-2, from=3-1]
\end{tikzcd}\]
For the fiber over the prime $q=2$ we have $O^2(K_4) = O^2(C_2) = e$, $O^2(A_4) = A_4$, and $O^2(C_3) = C_3$.  Thus the fiber over $q=2$ is the line
\[\begin{tikzcd}
	{\mf{p}_{e,2}} & {\mf{p}_{C_3,2}} & {\mf{p}_{A_4,2}}
	\arrow[to=1-1, from=1-2]
	\arrow[to=1-2, from=1-3]
\end{tikzcd}.\]
For the fiber over the prime $q=3$ we have $O^3(K_4) =K_4$, $ O^3(C_2) = C_2$, $O^3(A_4) = K_4$ and $O^2(C_3) = 3$.  Thus the fiber over $q=2$ is the line
\[\begin{tikzcd}
	{\mf{p}_{e,3}} & {\mf{p}_{C_2,3}} & {\mf{p}_{K_4,3}}
	\arrow[to=1-1, from=1-2]
	\arrow[to=1-2, from=1-3]
\end{tikzcd}.\]
\end{example}

\begin{example}
    We end with our most complicated example $G = \mathrm{PSL}_2(\mathbb{F}_7)\cong \mathrm{GL}_3(\mathbb{F}_2)$.  This group is notable for being the automorphism group of the Fano plane.  It is also the second smallest non-abelian finite simple group.  The order of this group is $|G| = 168 = 2^3\cdot 3\cdot 7$.  The subconjugacy lattice can be found in the GroupNames database \cite{GroupNames}; see also \cite[pp. 207-212]{DF91}. The fiber over the generic prime $p\neq 2,3,7$ is
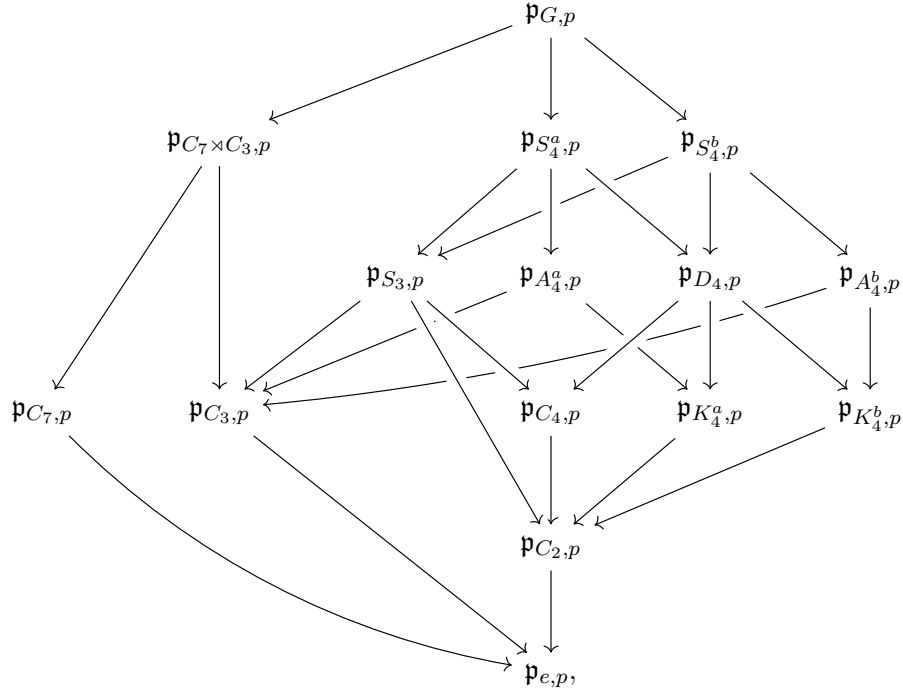
\[\begin{tikzcd}[row sep=1.1cm]
	&&& {\mf{p}_{G,p}} \\
	& {\mf{p}_{C_7\rtimes C_3,p}} && {\mf{p}_{S_4^a,p}} & {\mf{p}_{S_4^b,p}} \\
	&& {\mf{p}_{S_3,p}} & {\mf{p}_{A_4^a,p}} & {\mf{p}_{D_4,p}} & {\mf{p}_{A_4^b,p}} \\
	{\mf{p}_{C_7,p}} & {\mf{p}_{C_3,p}} && {\mf{p}_{C_4,p}} & {\mf{p}_{K_4^a,p}} & {\mf{p}_{K_4^b,p}} \\
	&&& {\mf{p}_{C_2,p}} \\
	&&& {\mf{p}_{e,p},}
	\arrow[from=1-4, to=2-2]
	\arrow[from=1-4, to=2-4]
	\arrow[from=1-4, to=2-5]
	\arrow[from=2-2, to=4-1]
	\arrow[from=2-2, to=4-2]
	\arrow[from=2-4, to=3-3]
	\arrow[from=2-5, to=3-5]
	\arrow[from=2-5, to=3-6]
	\arrow[from=3-6, to=4-6]
	\arrow[curve={height=24pt}, from=4-1, to=6-4]
	\arrow[from=4-2, to=6-4]
	\arrow[from=4-4, to=5-4]
	\arrow[from=4-5, to=5-4]
	\arrow[from=4-6, to=5-4]
	\arrow[from=5-4, to=6-4]
        \arrow[from=2-4, to=3-5]
        \arrow[from=2-5, to=3-3]
        \arrow[curve={height=-10pt}, from=3-6, to=4-2]
        \arrow[from=3-3, to=4-2, crossing over]
    \arrow[from=3-4, to=4-2, crossing over]        
	\arrow[from=3-3, to=4-4, crossing over]
	\arrow[from=3-3, to=5-4, crossing over]
	\arrow[from=3-4, to=4-5, crossing over]
        \arrow[from=3-5, to=4-4, crossing over]
	\arrow[from=3-5, to=4-5, crossing over]
	\arrow[from=3-5, to=4-6, crossing over]
        \arrow[crossing over, from=2-4, to=3-4]
        \arrow[crossing over, from=2-4, to=3-5]
\end{tikzcd}\]
where we have written $S^a_{4}$ and $S^b_{4}$ to differentiate between the two non-conjugate copies of $S_4$, and similarly for $A_4^a$, $A_4^b$, $K_4^a$ and $K_4^b$. The following table displays the values of $O^q(H)$ for $p=2,3,7$ and $H\leq G$.
\begin{center}
\begin{tabular}{ |c||c|c|c| } 
 \hline
 $H\leq G$ & $O^2(H)$ & $O^3(H)$ & $O^7(H)$ \rule{0pt}{2.2ex} \\ \hline \hline
 $C_7\rtimes C_3$ & $C_7\rtimes C_3$ & $C_7$ & $C_7\rtimes C_3$ \\ \hline
 $S_4^a$ & $A_4^a$ & $S_4^a$ &  $S_4^a$\\ \hline
 $S_4^b$ & $A_4^b$ & $S_4^b$ &  $S_4^b$\\ \hline
 $C_7$ & $C_7$ & $C_7$ & $e$ \\ \hline
 $S_3$ & $C_3$ & $S_3$  & $S_3$\\ \hline
 $A_4^a$ & $A_4^a$ & $K_4^a$ & $A_4^a$ \\ \hline
 $A_4^b$ & $A_4^b$ & $K_4^b$ & $A_4^b$ \\ \hline
 $D_4$ & $e$ & $D_4$ &  $D_4$ \\ \hline
 $C_3$ & $C_3$ & $e$  & $C_3$ \\ \hline
 $C_4$ & $e$ & $C_4$ &  $C_4$ \\ \hline
 $K_4^a$ & $e$ & $K_4^a$  & $K_4^a$ \\ \hline
 $K_4^b$ & $e$ & $K_4^b$ & $K_4^b$\\ \hline
 $C_2$ & $e$ & $C_2$  & $C_2$ \\ \hline
 $e$ & $e$ & $e$ & $e$ \\ \hline
\end{tabular}
\end{center}

\noindent
The fiber over $p=2$ is 
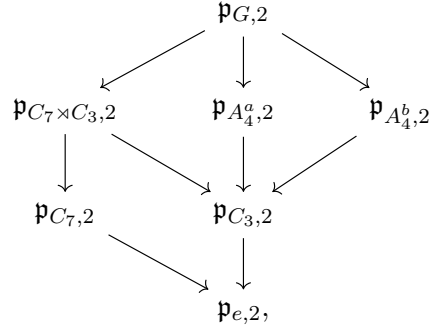
\[\begin{tikzcd}
	& {\mf{p}_{G,2}} \\
	{\mf{p}_{C_7\rtimes C_3,2}} & {\mf{p}_{A_4^a,2}} & {\mf{p}_{A_4^b,2}} \\
	{\mf{p}_{C_7,2}} & {\mf{p}_{C_3,2}} \\
	& {\mf{p}_{e,2},}
	\arrow[from=1-2, to=2-1]
	\arrow[from=1-2, to=2-2]
	\arrow[from=1-2, to=2-3]
	\arrow[from=2-1, to=3-1]
    \arrow[from=2-1, to=3-2]
	\arrow[from=2-2, to=3-2]
	\arrow[from=2-3, to=3-2]
	\arrow[from=3-1, to=4-2]
	\arrow[from=3-2, to=4-2]
\end{tikzcd}\]
the fiber over $p=3$ is 
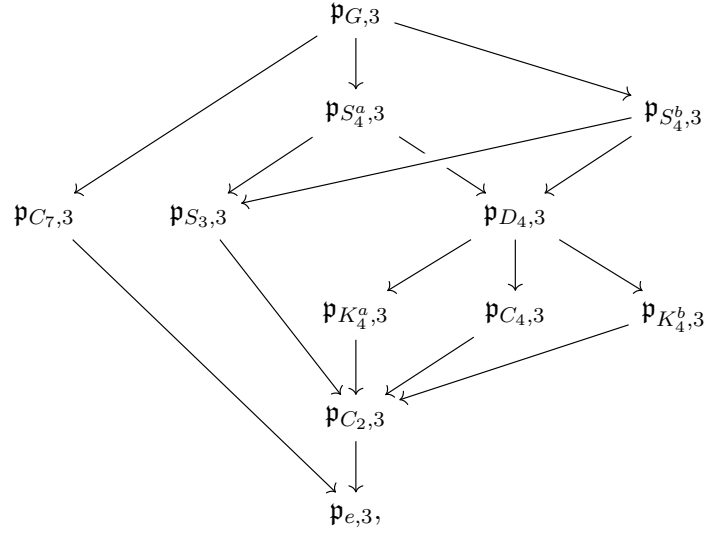
\[\begin{tikzcd}
	&& {\mf{p}_{G,3}} \\
	&& {\mf{p}_{S_4^a,3}} && {\mf{p}_{S_4^b,3}} \\
	{\mf{p}_{C_7,3}} & {\mf{p}_{S_3,3}} && {\mf{p}_{D_4,3}} \\
	&& {\mf{p}_{K_4^a,3}} & {\mf{p}_{C_4,3}} & {\mf{p}_{K_4^b,3}} \\
	&& {\mf{p}_{C_2,3}} \\
	&& {\mf{p}_{e,3},}
	\arrow[from=1-3, to=2-3]
	\arrow[from=1-3, to=2-5]
	\arrow[from=1-3, to=3-1]
	\arrow[from=2-3, to=3-2]
	\arrow[from=2-3, to=3-4]
	\arrow[from=2-5, to=3-4]
	\arrow[from=3-1, to=6-3]
	\arrow[from=3-2, to=5-3]
	\arrow[from=3-4, to=4-3]
	\arrow[from=3-4, to=4-4]
	\arrow[from=3-4, to=4-5]
	\arrow[from=4-3, to=5-3]
	\arrow[from=4-4, to=5-3]
	\arrow[from=4-5, to=5-3]
	\arrow[from=5-3, to=6-3]
        \arrow[crossing over,from=2-5, to=3-2]
\end{tikzcd}\]
and the fiber over $p=7$ is 
\[\begin{tikzcd}[row sep=1.1cm]
	&&& {\mf{p}_{G,7}} \\
	& {\mf{p}_{C_7\rtimes C_3,7}} && {\mf{p}_{S_4^a,7}} & {\mf{p}_{S_4^b,7}} \\
	&& {\mf{p}_{S_3,7}} & {\mf{p}_{A_4^a,7}} & {\mf{p}_{D_4,7}} & {\mf{p}_{A_4^b,7}} \\
	 & {\mf{p}_{C_3,7}} && {\mf{p}_{C_4,7}} & {\mf{p}_{K_4^a,7}} & {\mf{p}_{K_4^b,7}} \\
	&&& {\mf{p}_{C_2,7}} \\
	&&& {\mf{p}_{e,7}.}
		\arrow[from=1-4, to=2-2]
	\arrow[from=1-4, to=2-4]
	\arrow[from=1-4, to=2-5]
	\arrow[from=2-2, to=4-2]
	\arrow[from=2-4, to=3-3]
	\arrow[from=2-5, to=3-5]
	\arrow[from=2-5, to=3-6]
	\arrow[from=3-6, to=4-6]
	\arrow[from=4-2, to=6-4]
	\arrow[from=4-4, to=5-4]
	\arrow[from=4-5, to=5-4]
	\arrow[from=4-6, to=5-4]
	\arrow[from=5-4, to=6-4]
        \arrow[from=2-4, to=3-5]
        \arrow[from=2-5, to=3-3]
        \arrow[curve={height=-10pt}, from=3-6, to=4-2]
        \arrow[from=3-3, to=4-2, crossing over]
    \arrow[from=3-4, to=4-2, crossing over]        
	\arrow[from=3-3, to=4-4, crossing over]
	\arrow[from=3-3, to=5-4, crossing over]
	\arrow[from=3-4, to=4-5, crossing over]
        \arrow[from=3-5, to=4-4, crossing over]
	\arrow[from=3-5, to=4-5, crossing over]
	\arrow[from=3-5, to=4-6, crossing over]
        \arrow[crossing over, from=2-4, to=3-4]
        \arrow[crossing over, from=2-4, to=3-5]
\end{tikzcd}
\]
\end{example}

\printbibliography

\end{document}